\documentclass{amsart}

\usepackage{amssymb,amsmath}
\usepackage{mathrsfs,dsfont}
\usepackage[neveradjust]{paralist}
\usepackage{epic,eepic}
\usepackage{url}

\theoremstyle{definition}
\newtheorem{ntn}{Notation}[section]
\newtheorem{dfn}[ntn]{Definition}
\theoremstyle{plain}
\newtheorem{lem}[ntn]{Lemma}
\newtheorem{prp}[ntn]{Proposition}
\newtheorem{thm}[ntn]{Theorem}
\newtheorem{cor}[ntn]{Corollary}

\newtheorem{alg}[ntn]{Algorithm}
\theoremstyle{remark}

\newtheorem{rmk}[ntn]{Remark}
\newtheorem{exa}[ntn]{Example}

\newcommand{\boldzero}{{\mathbf{0}}}
\newcommand{\boldone}{{\mathbf{1}}}
\newcommand{\bolda}{{\mathbf{a}}}

\newcommand{\boldu}{{\mathbf{u}}}

\newcommand{\del}{\partial}
\newcommand{\eps}{\varepsilon}
\newcommand{\Fquot}[1]{F_{#1}/\hskip-0.5ex\sim_{\varphi_{{#1}}}}
\newcommand{\ideal}[1]{{\langle#1\rangle}}
\newcommand{\into}{\hookrightarrow}
\newcommand{\onto}{\twoheadrightarrow}

\newcommand{\dotPhi}{\dot\Phi}

\newcommand{\calA}{\mathscr{A}}
\newcommand{\calB}{\mathscr{B}}

\newcommand{\frakb}{{\mathfrak{b}}}
\newcommand{\frakm}{{\mathfrak{m}}}
\newcommand{\frakp}{{\mathfrak{p}}}

\newcommand{\CC}{\mathds{C}}
\newcommand{\KK}{\mathds{K}}
\newcommand{\NN}{\mathds{N}}
\newcommand{\PP}{\mathds{P}}
\newcommand{\QQ}{\mathds{Q}}

\newcommand{\ZZ}{\mathds{Z}}

\DeclareMathOperator{\depth}{depth}

\DeclareMathOperator{\gr}{gr}
\DeclareMathOperator{\height}{height}
\DeclareMathOperator{\Hom}{Hom}

\DeclareMathOperator{\ini}{in}
\DeclareMathOperator{\lcm}{lcm}

\DeclareMathOperator{\rk}{rk}
\DeclareMathOperator{\Spec}{Spec}

\begin{document}

\title[Newton graded toric rings]
{Cohen-Macaulayness and computation\\ of Newton graded toric rings}

\author{Mathias Schulze}
\address{
M. Schulze\\
Oklahoma State University\\
Dept.\ of Mathematics\\
401 MSCS\\
Stillwater, OK 74078\\
USA}
\email{mschulze@math.okstate.edu}
\thanks{MS was supported by the College of Arts \& Sciences at Oklahoma State University through a FY 2008 Dean's Incentive Grant.}

\author{Uli Walther}
\address{
U. Walther\\
Purdue University\\
Dept.\ of Mathematics\\
150 N.\ University St.\\
West Lafayette, IN 47907\\
USA}
\email{walther@math.purdue.edu}

\begin{abstract}
Let $H\subseteq \ZZ^d$ be a positive semigroup generated by $\calA\subseteq H$, 
and let $\KK[H]$ be the associated semigroup ring over a field $\KK$. 
We investigate heredity of the Cohen--Macaulay property from $\KK[H]$ 
to both its $\calA$-Newton graded ring and to its face rings.  We show 
by example that neither one inherits in general the Cohen--Macaulay property. 
On the positive side we show that for every $H$ there exist generating sets $\calA$ for which the Newton graduation preserves Cohen--Macaulayness. This gives an elementary proof for an important vanishing result on $A$-hypergeometric
Euler--Koszul homology. As a tool for our investigations we develop an algorithm 
to compute algorithmically the Newton filtration on a toric ring.
\end{abstract}

\subjclass{14M25,16W70}

\keywords{toric ring, Newton filtration, Cohen-Macaulay}

\maketitle
\tableofcontents
\numberwithin{equation}{section}

\section{Introduction}

The motivation for this article arises from the study of
the solutions of $A$-hypergeometric systems.
These systems of linear partial differential equations have been introduced by Gel'fand, Graev, Kapranov, and Zelevinski{\u\i} \cite{GGZ87,GKZ89} as a general framework including the classical hypergeometric functions.

Given an integer $d\times n$-matrix $A=((a_{i,j}))\in\ZZ^{d\times n}$, let $D_A$ be the (complex) Weyl algebra in the variables $x=x_1,\dots,x_n$ with corresponding differential operators $\del=\del_1,\dots,\del_n$ where $\del_i=\frac{\del}{\del x_i}$ for $i=1,\dots,n$. 
Writing $(\boldu_+)_j=\max(0,\boldu_j)$ and $\boldu_-=\boldu_+-\boldu$ for $\boldu\in\ZZ^n$, the \emph{toric relations} are defined by
\[
\Box_\boldu=\del^{\boldu_+}-\del^{\boldu_-}
\]
for $\boldu\in\ker A\cap\ZZ^n$. 
For a complex parameter vector $\beta\in\CC^d$, the induced \emph{$A$-hypergeometric system} is the system partial differential equations for $f(x_1,\ldots,x_n)$ given by all \emph{Euler equations} $E_i\bullet(f)=\beta_i\cdot f$ where
\[
E_i=\sum_{j=1}^n a_{i,j}x_i\del_i
\]
for $i=1,\dots,d$, and all \emph{toric equations} $\Box_\boldu\bullet(f)=0$ for $\boldu\in\ker A\cap\ZZ^n$.
More intrinsically, one considers the $A$-hypergeometric $D_A$-module
\[
M_A(\beta)=D_A/\ideal{E-\beta,I_A}
\]
where we abbreviate $E=E_1,\dots,E_d$ and 
\[
I_A=\ideal{\Box_\boldu\mid\boldu\in\ker A\cap\ZZ^n}\subseteq R_A=\CC[\del_1,\dots,\del_n].
\]
is the \emph{toric ideal}.
The quotient of $R_A$ by the toric ideal defines the \emph{toric ring} 
\[
S_A=R_A/I_A\cong\CC[t^{\bolda_1},\dots,t^{\bolda_n}]\subseteq\CC[t_1^{\pm1},\dots,t_d^{\pm1}].
\]

A fundamental tool for studying $M_\beta(A)$ is to consider it as the
$0$-th homology of the \emph{Euler--Koszul complex} of $E=E_1,\dots,E_d$ on $D\otimes_{R_A}S_A$, \cite{MMW05}.
Considered in a commutative polynomial ring $\CC[x,\del]$, the elements $E$ form a subset of a system of parameters on $S_A[x]$.
There are two main approaches to understanding the relationship of Euler--Koszul homology (and hence $M_A(\beta)$) with the commutative Koszul complex:
Adolphson \cite{Ado94} applied grading techniques with respect to the
\emph{Newton filtration} on $S_A$ using ideas of Kouchnirenko
\cite{Kou76}; on the other hand, \cite{MMW05} introduces homological methods on the category of \emph{toric modules}, a natural $\ZZ^d$-graded category that contains all face rings of $S_A$ and their $\ZZ^d$-shifts.
Two natural questions arise:
\begin{enumerate}[(Q1)]
\item\label{q1} Is the Newton graded ring of a Cohen-Macaulay toric ring Cohen-Macaulay?
\item\label{q2} Are the face rings of a Cohen-Macaulay toric ring Cohen-Macaulay?
\end{enumerate}
Closer inspection shows that (Q\ref{q1}) does not quite make sense as
stated: speaking of a Newton filtration requires a distinguished (finite) set
of generators on the toric ring in order to form the Newton polyhedron. The (historic) default choice for $S_A$ is the column set of $A$. 
In that case, Okuyama claims a positive answer to (Q\ref{q1}) in \cite[Prop.~3.1]{Oku06a} and \cite[Prop.~3.4]{Oku06b} and derives certain vanishing properties
for Euler--Koszul homology. We shall show that these claims are
incorrect. However, we also show that for suitable choices of
generating sets the answer to (Q\ref{q1}) is positive, 
we prove the vanishing of Euler--Koszul homology on Cohen--Macaulay rings, 
and we describe an algorithm to compute the Newton filtration induced
by arbitrary (finite) generating sets.

Regarding (Q\ref{q2}) we will show that in general the answer is
negative as well, and elaborate on a criterion discussed in
\cite{HT86} regarding Cohen--Macaulayness of toric rings.

The following notation will be used throughout this note.

\begin{ntn}\label{49}
We let $\NN=\{0,1,2,\ldots\}$ be the natural numbers and denote by
$\QQ_+$ the non-negative rational numbers.

Let $H$ be an \emph{affine semigroup}.
By this we mean a finitely generated, commutative, torsion-free, cancelative monoid.
We assume that $H$ is \emph{positive} which means that $0$ is the only invertible element.
By Hochster's transformation \cite[\S2]{Hoc72}, $H$ may be viewed as a submonoid of $\NN^d$ for some $d\in\NN$.
The Grothendieck group $\ZZ H$ of $H$ is isomorphic to $\ZZ^r$ where $r=\rk H$ and we can view $H\subseteq\ZZ H$.
We write $\KK H$ for the $\KK$-vector space $\KK\otimes_\ZZ\ZZ H$ and identify $h\in H$ with $1\otimes h\in\KK H$ if $\KK$ has characteristic zero.

Let $\KK$ be an arbitrary field and $\KK[H]$ the semigroup ring of $H$ over $\KK$. 
The inclusion $H\subseteq\ZZ H\cong\ZZ^r$ (or $H\subseteq\NN^d$) gives rise to an inclusion $\KK[H]\subseteq\KK[\ZZ^r]\cong\KK[t_1^{\pm1},\dots,t_r^{\pm1}]$ (or
$\KK[H]\subseteq\KK[\NN^d]\cong\KK[t_1,\dots,t_d]$). 
This makes $\Spec(\KK[H])$ a toric variety and $\KK[H]$ a \emph{toric ring} endowed with a canonical $\ZZ H$-grading given by $\deg(t^h)= h$ for $h\in H$. 
Generally, for a subset $S\subseteq \ZZ^d$, we let $\KK[S]:=\bigoplus_{s\in S}\KK\cdot t^s$.
We denote by $\frakm=\frakm_H=\KK[H\smallsetminus\{0\}]$ the unique maximal $\ZZ H$-graded ideal which defines the vertex of $\Spec(\KK[H])$.

Generally, $\calA$ will be a finite set of semigroup generators for $H$. 
Then $\calA$ defines a presentation 
\[
R_\calA=\KK[(y_a)_{a\in\calA}]=\KK[\NN^\calA]\onto\KK[H]=S_\calA
\]
with kernel $I_\calA$ and we write $y_i=y_{a_i}$ for $i=1,\dots,n$ if $\calA=\{a_1,\dots,a_n\}$. 
In the situation of hypergeometric systems, the given matrix
$A=(a_1,\dots,a_n)$ replaces the generating set $\calA$ and we write
$\del_i$ rather than $y_i$.

Consider the rational projective $r$-space 
\[
\PP_\QQ^H=\PP_\QQ(\QQ H\times \QQ)
\]
containing $\QQ H$ via the embedding $qh\mapsto (qh:1)$. Since $H$ is
positive, there is a linear functional $h\in\Hom_\QQ(\QQ H,\QQ)$ such
that $h^{-1}(0)\cap\QQ_+ H=\{0\}$ while $h(H)\geq 0$. 
Pick a rational weight vector $L\colon \calA\to\QQ$ and choose
$0>\eps\in\QQ$ such that $L(a)<0$ implies $\eps>h(a)/L(a)$. 
This can be done since $\calA$ is finite, cf.~\cite{SW05}. 
The convex hull in $\PP_\QQ^H\smallsetminus
h^{-1}(-\eps)$ of the origin $0=(0:1)$ and the points $a^L:=(a:L(a))$,
$a\in\calA$, is the \emph{$(\calA,L)$-polyhedron}. 
Its simplicial complex $\Phi^L_\calA$ of faces not containing $(0:1)$ is the
\emph{$(\calA,L)$-umbrella}, cf.~\cite[\S 2.3]{SW05}.
The dimension $\dim(\tau)$ of $\tau\in\Phi^L_\calA$ is its topological dimension as
a boundary component of $\Delta^L_\calA$. 
We denote by $\Phi^{L,k}_\calA\subseteq\Phi^L_\calA$ the subcollection of faces of
dimension $k$, hence a \emph{facet} is of dimension $r-1$. 
We reserve $\sigma$ for elements of $\Phi^{L,r-1}_\calA$ and $\tau$ for general
elements in $\Phi^L_\calA$.

The\emph{interior $(\calA,L)$-umbrella} $\dotPhi^L_\calA$ consists of the faces
$\tau\in\Phi^L_\calA$ for which there is no facet $\sigma$ of
$\Delta^L_\calA$ with $\tau\subseteq\sigma\ni 0$. One may view
$\dotPhi^L_\calA$ as the quotient complex of $\Phi^L_A$ by the
subcomplex of faces contained in its boundary as a piecewise linear
$(r-1)$-manifold. 

Consider the special weight vectors $\boldzero=(0)_{a\in\calA}$ and
$\boldone=(1)_{a\in\calA}$. Then $\Delta^\boldone_\calA$ is the convex
hull of $\calA$ and $0$, while $\Delta^\boldzero_\calA$ is the closure
in $\PP_\QQ^H$ of the positive rational cone $C_H=\QQ_+ H$. 
One can identify $\Phi^\boldzero_\calA$ with the simplicial complex of a cross section of the cone $C_H$. 
If $\calA$ is contained in a hyperplane in $\QQ H$ then $\KK[H]$ is
the coordinate ring of a projective variety and we say that $\calA$ is
\emph{projective}. 

Let $\tau\in\Phi^L_\calA$. We call the cone $C_\tau=\QQ_+\tau$ over
$\tau$ a \emph{face cone} in general and \emph{facet cone} if $\tau$
is a facet. We abbreviate $H_\tau=H\cap C_\tau$ and
$\calA_\tau=\calA\cap C_\tau$, and we denote by $\frakm_\tau$ the
maximal $\ZZ H$-graded ideal $\KK[H_\tau\smallsetminus\{0\}]$ in
$\KK[H_\tau]$.  For a facet $\sigma\in\Phi^{L,r-1}_\calA$, we denote
by $\ell_\sigma$ the unique linear form on $\QQ H$ such that
$\max\ell_\sigma(\Delta^\boldone_\calA)=1$ and
$\ell_\sigma^{-1}(1)\cap\Delta^\boldone_\calA=\sigma$.
\end{ntn}

\section{A necessary criterion for Cohen--Macaulayness of toric rings}

In this section, we discuss Cohen--Macaulayness of affine semigroup
rings following Hoa and Trung \cite{HT86,HT88}. We give a simplified
proof for one of their results and use it to find a
negative answer to Question (Q\ref{q2}).

The following result is classical and a special case of Proposition~\ref{17}.

\begin{lem}\label{9}
The set of $\ZZ H$-graded primes in $\KK[H]$ is in bijection with
$\Phi^\boldzero_\calA$. 
Each $\tau\in \Phi^\boldzero_\calA$ corresponds to a prime ideal $\frakp_\tau=\KK[H\smallsetminus C_\tau]$.\qed
\end{lem}

\begin{dfn}\label{10}
Let $\Bar H$ be the set of $h\in\ZZ H$ for which there exists
$\calB\subseteq H$, $0\notin \calB$, such that
$\calB_\tau\ne\emptyset$ for all $\tau\in\Phi^{\boldzero,r-2}_\calA$
while $h+\calB\subseteq H$.
\end{dfn}

For example, if $H=2\NN+3\NN\subseteq \NN$ then $\bar H=\ZZ$, where
$k\in\bar H$ is witnessed by $\calB_k=\{3-k\}$ and the equation
$k+(3-k)=3\in H$. 
Note that these relations encode cosets $t^k+\KK[H]$ in $H^1_{\ideal{\calB}}(\KK[H])=\KK[\ZZ H]/\KK[H]$ which are non-zero for $k\le1$.  

Note that $\bar H\supseteq H$ and if $r\ge2$ then $\bar H\subseteq C_H$.
Indeed, for all $h\in\bar H$ and for any $\tau\in\Phi^{\boldzero,r-2}_\calA$ there
is a $b\in\calB\cap C_\tau$ with $h+b\in H$. 
So any linear functional that is non-negative on $H$ and vanishes on $C_\tau$ must be non-negative on $h$.

Part \eqref{14c} of the following result (cf.~\cite[Cor.~2.2]{HT86}, \cite[Lem.~4.3]{HT88}) states a necessary condition for $\KK[H]$ to be Cohen--Macaulay; our contribution is a simplified proof. 
When combined with the topological description of the Cohen--Macaulayness of $\KK[\Bar
H]$ from \cite{HT86,HT88}, this characterizes Cohen--Macaulayness of $\KK[H]$.

\begin{prp}\label{14}
Let $\calB\subseteq H$ with $0\not\in\calB\neq\emptyset$ and put $\frakb=\ideal{t^b\mid b\in\calB}\subseteq \KK[H]$.
\begin{enumerate}
\item\label{14a} $H^1_\frakm(\KK[H])\subseteq\KK[\Bar H]/\KK[H]$ and $H^1_\frakm(\KK[H])\subseteq H^1_\frakb(\KK[H])$.
\item\label{14b} For $h\in\ZZ H\backslash H$ with $h+\calB\subseteq H$, $h$ induces a non-zero element in $H^1_\frakb(\KK[H])$.
In particular, $\KK[\Bar H]/\KK[H]\subseteq\sum_\calB H^1_\frakb(\KK[H])$ with $\calB$ as in Definition \ref{10}.
\item\label{14c} If $\KK[H]$ is Cohen--Macaulay of dimension $r\ge2$ then $H=\Bar H$.
\item\label{14d} For $\KK[H]$ of dimension $r=2$, $\Bar H/H=H^1_\frakm(\KK[H])$.
\end{enumerate}
\end{prp}

\begin{proof}
Any $\ZZ H$-graded element of $H^1_\frakm(\KK[H])$ is the coset modulo $\KK[H]$ of an element $(t^{h_1-c_1a_1},\ldots,t^{h_k-c_ka_k})$ where $a_1,\ldots,a_k\in\calA$ generate $H$ and where $c_i\in\NN$ and $h_i\in H$ with $h_1-c_1a_1=h_j-c_ja_j$ for all $j$.
So $h_1-c_1a_1\in\Bar H$ with $\calB=\{a_1,\dots,a_k\}$ which proves the first inclusion.

As $\KK[H]$ is a domain, $H^1_\frakb(\KK[H])=\left(\bigcap_{b\in\calB}\KK[H+\ZZ b]\right)/\KK[H]$ while $H^1_\frakm(\KK[H])=\left(\bigcap_{a\in\calA}\KK[H+\ZZ a]\right)/\KK[H]$. 
Pick $b\in\calB$. 
Then $0\ne b=\sum_{a\in\calA}k_aa$ with $k_a\in\NN$ implies $H+\ZZ b\supseteq H+\ZZ a$ for any $a$ with $k_a> 0$. 
Thus, $\bigcap_{b\in\calB}\KK[H+\ZZ b]$ contains $\bigcap_{a\in\calA} \KK[H+\ZZ a]$ and the first claim follows.

By definition $(\KK[H]+\KK[H]\cdot t^h)/\KK[H]$ is $\frakb$-torsion. 
Applying $R\Gamma_\frakb$ to the short exact sequence 
\[
0\to\KK[H]\to\KK[H]+\KK[H]\cdot t^h\to(\KK[H]+\KK[H]\cdot t^h)/\KK[H]\to 0
\]
yields
\[
\underbrace{\Gamma_\frakb(\KK[H]+\KK[H]\cdot t^h)}_{=0}
\to
\underbrace{(\KK[H]+\KK[H]\cdot t^h)/\KK[H]}_{\ni t^h+\KK[H]}
\to H^1_\frakb(\KK[H]),
\]
which proves the second claim.

By definition, $\KK[\Bar H]/\KK[H]$ is $\frakb$-torsion. 
Applying $R\Gamma_\frakb$ to the short exact sequence
\[
0\to\KK[H]\to\KK[\Bar H]\to\KK[\Bar H]/\KK[H]\to 0
\]
yields
\[
0=\Gamma_\frakb(\KK[\Bar H])
\to\KK[\Bar H]/\KK[H]
\to H^1_\frakb(\KK[H]),
\]
which proves the second claim.

For the third claim assume that $\KK[H]$ is Cohen--Macaulay with
$\dim(\KK[H])\ge2$. 
Let $\calB\subseteq H$ with $0\notin \calB\not =\emptyset$. 
Then $\frakb\ne0$ and hence $\dim(\KK[H]/\frakb)<r=\dim(\KK[H])$ as $\KK[H]$ is a domain. 
Since $\frakb$ is monomial, it is $\ZZ^d$-graded. So by Lemma \ref{9} any associated prime of $\frakb$ is of the form $\frakp_\tau$ with $\tau\in\Phi^\boldzero_\calA$. 
Hence there is an inclusion
\[
\KK[H_\tau]=\KK[H]/\frakp_\tau\into\KK[H]/\frakb.
\]
Now let $h\in\bar H$ and consider a $\calB$ corresponding to $h$ in Definition~\ref{10}. 
Then $\frakb$ contains a non-zerodivisor on $\KK[H_{\tau'}]$ for every $\tau'\in\Phi^{\boldzero,r-2}_\calA$. 
On the other hand, $\frakb$ kills $\KK[H_\tau]\subseteq\KK[H]/\frakb$, so $\dim(\tau)\le r-3$. 
This being so for all associated primes of $\frakb$, $\height\frakb\geq 2$. 
As $\KK[H]$ is Cohen--Macaulay, $2\le\height\frakb=\depth\frakb$ and hence $H^1_\frakb(\KK[H])=0$. 
It follows from \eqref{14b} that we therefore must have $h\in H$ and so $\bar H=H$ as claimed.

Finally assume that $r=2$. 
Then any $\calB$ as in Definition \ref{10} meets every $C_\tau$,  $\tau\in\Phi_\calA^{\boldzero,0}$, and in particular the extremal rays of $C_H$.
This means that $\frakb$ is $\frakm$-primary and hence $\KK[\Bar H]/\KK[H]\subseteq H_\frakm^1(\KK[H])$ by \eqref{14b}.
Then the last claim follows from the first inclusion in \eqref{14a}.
\end{proof}

\begin{rmk}
In fact, $\bar H$ is a subsemigroup of $\ZZ H$ and if $r=2$ then $\KK[\bar H]$ is
the ideal transform of $\KK[H]$ relative to $\frakm$, cf.~\cite{BS}.
\end{rmk}

Investigating cases where the criterion of Proposition \ref{14} applies to facets, but not all, of $H$, we found the following example showing that the answer to Question (Q\ref{q2}) is negative: 
the facet rings of a Cohen--Macaulay toric ring are not Cohen--Macaulay in general, even in the projective case.

\begin{exa}\label{23}
Consider the affine semigroup $H$ generated by 
\begin{equation}\label{30}
\calA=\{a_1,\dots,a_6\}=
\left\{
\begin{pmatrix}1\\0\\0\end{pmatrix},
\begin{pmatrix}1\\1\\0\end{pmatrix},
\begin{pmatrix}1\\3\\0\end{pmatrix},
\begin{pmatrix}1\\4\\0\end{pmatrix},
\begin{pmatrix}1\\1\\1\end{pmatrix},
\begin{pmatrix}1\\4\\1\end{pmatrix}
\right\}.
\end{equation}
\begin{figure}[ht]
\caption{The $(\calA,\boldone)$-polyhedron for Example~\ref{23}}
\begin{center}
\setlength{\unitlength}{0.65mm}
\setlength{\unitlength}{1mm}
\begin{picture}(100,35)(0,0){
\put(20,5){
\put(0,0){\vector(1,0){90}}
\put(0,0){\vector(0,1){30}}
\put(0,0){\vector(-1,-1){7}}
\multiput(0,0)(20,0){5}{\path(0,1)(0,-1)}
\multiput(0,0)(0,20){2}{\path(1,0)(-1,0)}
\put(-2,0){\makebox(0,0)[br]{$0$}}
\put(20,-2){\makebox(0,0)[t]{$1$}}
\put(40,-2){\makebox(0,0)[t]{$2$}}
\put(60,-2){\makebox(0,0)[t]{$3$}}
\put(80,-2){\makebox(0,0)[t]{$4$}}
\put(89,-1){\makebox(0,0)[t]{$t_2$}}
\put(-2,20){\makebox(0,0)[r]{$1$}}
\put(-1,29){\makebox(0,0)[r]{$t_1$}}
\put(-7,-6){\makebox(0,0)[br]{$t_3$}}
\put(0,0){\texture{
88888888 88888888 88888888 88888888 88888888 88888888 88888888 88888888 
88888888 88888888 88888888 88888888 88888888 88888888 88888888 88888888 
88888888 88888888 88888888 88888888 88888888 88888888 88888888 88888888 
88888888 88888888 88888888 88888888 88888888 88888888 88888888 88888888}
\shade\path(0,20)(80,20)(66,6)(6,6)(0,20)\thicklines\path(0,20)(80,20)(66,6)(6,6)(0,20)}
\put(0,20){\circle*{1}\put(1,1){\makebox(0,0)[bl]{$a_1$}}}
\put(20,20){\circle*{1}\put(0,1){\makebox(0,0)[b]{$a_2$}}}
\put(60,20){\circle*{1}\put(0,1){\makebox(0,0)[b]{$a_3$}}}
\put(80,20){\circle*{1}\put(0,1){\makebox(0,0)[b]{$a_4$}}}
\put(6,6){\circle*{1}\put(1,1){\makebox(0,0)[bl]{$a_5$}}}
\put(66,6){\circle*{1}\put(1,0){\makebox(0,0)[tl]{$a_6$}}}
\dottedline{1}(0,0)(24,6)
\dottedline{2}(24,6)(80,20)
\dottedline{1}(0,0)(6,6)
\dottedline{1}(0,0)(66,6)
}
}
\end{picture}
\end{center}
\end{figure}
Note that $\ell=(1,0,0)$ equals $1$ on $\calA$, so  $H$ is projective.
The linear form $\ell_1=(0,0,1)$ defines the face
$\tau_1\in\Phi^{\boldzero,1}_\calA$ spanned by
$\calA'=\{a_1,a_2,a_3,a_4\}$; the corresponding submonoid $H_{\tau_1}$ of $H$ is $\NN\calA'$.
With $\calB=\{a_1,a_4\}$,
\[
2a_3-a_4=2a_2-a_1=(1,2,0)\in\Bar H_{\tau_1}\smallsetminus H_{\tau_1}
\]
and so $\KK[H_{\tau_1}]$ is not Cohen-Macaulay, as is well-known.

We now test whether $(1,2,0)\in\bar H$.
Inspection shows that $\Phi^{\boldzero,1}_\calA=\{\tau_1,\dots,\tau_4\}$ where
\[
\tau_1\ni a_1,a_2,a_3,a_4,\quad\tau_2\ni a_1,a_5,\quad\tau_3\ni a_4,a_6,\quad\tau_4\ni a_5,a_6
\]
are defined by the functionals
\[
\ell_1=(0,0,1),\ \ell_2=(0,1,-1),\ \ell_3=(1,-1/4,0),\ \ell_4=(1,0,-1).
\]
So any set $\calB$ to be used in Definition \ref{10} must contain an element $b\in\NN a_5+\NN a_6$.
Consider a relation $(1,2,0)+kb\in H$ with $k\in\NN$.
Note that $\ell_1(a_5)=\ell_1(a_6)$ while $\ell_1(\tau_1)=0$ and so $a_5$
and $a_6$ must appear with opposite coefficients. 
Moving the $a_5$- and $a_6$-terms to one side, one obtains a relation $(1,2,0)+k(a_6-a_5)\in H_{\tau_1}$ where now $k\in\ZZ$.
As $H_{\tau_1}\cap\ell^{-1}_1(1)=\calA'$ we find that
$(1,2,0)+k(a_6-a_5)=(1,2,0)+k(0,3,0)\in\calA'$ which is clearly
impossible. It follows that $(1,2,0)\notin\bar H$ and so
$H^1_\frakm(\KK[H])$ is zero in degree $(1,2,0)$. 

We now show that $\KK[H]$ is actually a Cohen--Macaulay ring. To this
end consider the ideal $J$ of $\KK[H]$ given by
\begin{equation}\label{48}
J=\ideal{I_{\calA'},y_4y_5-y_2y_6,y_3y_5-y_1y_6}.	
\end{equation}
Obviously, $J$ is contained in the toric ideal $I_\calA$, and
$J+\ideal{y_1,y_4+y_5,y_6}$ is $\frakm$-primary since its residue ring
is spanned by the monomials
\[
1,\quad y_2,\quad y_3,\quad y_5,\quad y_2^2,\quad y_2y_5,\quad y_3^2.
\]
Hence $\dim(\KK[H]/J)=3$ and so $I_\calA$ is a minimal prime of
$J$. Below we will show that $R_\calA/J$ is Cohen--Macaulay. This implies
that every associated prime of $J$ is of dimension $3$. Thus,
$J=I_\calA$ provided that the two ideals have the same degree. 
The simplicial volume of $\calA\cup\{0\}$ is $7$ and equals the degree of $I_\calA$. On the other hand, $\deg(R_\calA/J)=\deg(R_\calA/(J+\ideal{y_1,y_4+y_5,y_6}))$ is also $7$, and it follows that $J=I_\calA$.

We proceed to showing that $J$ is a Cohen--Macaulay ideal. 
The following is certainly Gr\"obner folklore but we don't know a reference. 

\begin{lem}
Suppose $J$ is an ideal in a polynomial ring $R=\KK[y_1,\ldots,y_n]$ and let $\le$ be a term order on $R$. 
Write $\ini_\le(J)$ for the initial ideal of $J$ under $\le$. 
If $y_n$ is not a zerodivisor on $R/\ini_\le(J)$
then $y_n$ is a non-zerodivisor on $R/J$.
\end{lem}

\begin{proof}
Let $G$ be a reduced $\le$-Gr\"obner basis for $J$ and suppose $y_nf\in J$ for some $f\in R$. 

If $f-f'\in J$ for some second element $f'\in R$ then $y_nf\in J$ if and only if $y_nf'\in J$ while of course $f\in J$ if and only if $f'\in J$. In particular, we may assume that $f$ is equal to its $\le$-normal form relative to $G$.

As $y_nf\in J$ we have $\ini_\le(y_nf)\in\ini_\le(J)$. 
The hypothesis implies that $\ini_\le(f)\in\ini_\le(J)$. 
Hence either $f=0$, or $f$ can be $\le$-reduced relative to $G$. As $f$ is in normal form, $f=0$.
\end{proof}

The lemma implies that $y_6$ is a non-zerodivisor on $R_\calA/J$. Indeed, one may verify by hand that the four generators
\[
y_2y_3-y_1y_4,\quad y_2^3-y_1^2y_3,\quad y_3^3-y_2y_4^2,\quad
y_3^2y_1-y_2^2y_4
\]
of $I_{\calA'}$ together with the two further generators of $J$ in \eqref{48} form a Gr\"obner basis for $J$ under the graded
reverse-lexicographic order. Since the initial terms of $J$ do not involve $y_6$ the desired conclusion follows.

It hence suffices to show that the two-dimensional quotient
\[
S'=R_\calA/(J+\ideal{y_6})\cong \KK[H_{\tau_1}][y_5]/\ideal{y_3y_5,y_4y_5}
\]
has vanishing $0$-th and $1$-st local cohomology with respect to the maximal ideal $\frakm'=\ideal{y_1,\ldots,y_5}$ of $S'$. 
The decomposition $\ideal{y_3y_5,y_4y_5}=\ideal{y_5}\cap\ideal{y_3,y_4}$ in $\KK[H_{\sigma_1}][y_5]$ gives a short exact sequence
\[
0\to\frac{\KK[H_{\tau_1}][y_5]}{\ideal{y_3y_5,y_4y_5}}\to
\frac{\KK[H_{\tau_1}][y_5]}{\ideal{y_5}}\oplus\frac{\KK[H_{\tau_1}][y_5]}{\ideal{y_3,y_4}}
\to\frac{\KK[H_{\tau_1}][y_5]}{\ideal{y_3,y_4,y_5}}\to 0.
\]
Note that
$\KK[H_{\tau_1}][y_5]/\ideal{y_3,y_4}\cong\KK[y_1,y_2,y_5]/\ideal{y_2^3}$
so that the two rightmost displayed rings are complete intersections. 
As $\KK[H_{\tau_1}]$ is a domain, the long exact local cohomology sequence with support in $\frakm'$ reads
\begin{gather*}
0\to H^0_{\frakm'}(S')\to 0\oplus 0\to\\
0\to H^1_{\frakm'}(S')\to H^1_{\frakm'}\left(\frac{\KK[H_{\tau_1}][y_5]}{\ideal{y_5}}\right)\oplus0\to H^1_{\frakm'}\left(\frac{\KK[H_{\tau_1}][y_5]}{\ideal{y_3,y_4,y_5}}\right)\to \cdots	
\end{gather*}
Our claim then follows if we can show that
\[
H^1_{\frakm'}(\KK[H_{\tau_1}][y_5]/\ideal{y_5})\to
H^1_{\frakm'}(\KK[H_{\tau_1}][y_5]/\ideal{y_3,y_4,y_5})
\]
is injective. 
However,
$H^1_{\frakm'}(\KK[H_{\tau_1}][y_5]/\ideal{y_5})\cong\KK\cdot(y_2^2/y_1,y_3^2/y_4,0)$,
generated by the indicated $1$-cocycle in the \v Cech complex to $y_1,y_4,y_5$ on
$\KK[H_{\tau_1}][y_5]/\ideal{y_5}$. 
Modulo $y_3$ and $y_4$ this becomes the class of $y_2^2/y_1$ in the
\v Cech complex of $y_1$ on
$\KK[H_{\tau_1}]/\ideal{x_3,x_4}\cong \KK[y_1,y_2]/\ideal{y_2^3}$, and that is clearly non zero.

Of course, all claims made are corroborated by computer results.
\end{exa}

\section{An algorithm to compute the Newton filtration}\label{NF}

A rational polyhedron $0\in\Delta\subseteq C_H$ defines a Newton filtration $k\cdot\Delta\cap H$, $k\in\QQ$, on $H$ and hence on $\KK[H]$.
This filtration is separated if $\Delta$ is bounded; it is exhaustive if $k\cdot\Delta\cap H$ generates $H$ for some $k$.
We shall impose both conditions and assume that $\calA=\Delta\cap H$ is a finite set of generators of $H$. 
To recover the Newton polytope $\Delta=\Delta_\calA^\boldone$ from $\calA$ we shall use Gr\"obner methods.

The main goal of this section to develop an algorithm to compute the associated graded ring to such a Newton filtration on $\KK[H]$. 
This algorithm is based on Gr\"obner bases and represents our computational tool to approach Question (Q\ref{q1}).

We start with a formal definition of the Newton filtration relative to $\calA$. 
Recall that $\ell_\sigma$ is the linear functional associated to the facet $\sigma\in\Phi^{\boldone,r-1}_\calA$. 

\begin{dfn}\label{36}
For $h\in H$ its \emph{(Newton) $\calA$-degree} is 
\[
\deg_\calA(h)=\max\{\ell_\sigma(h)\mid \sigma\in\Phi^{\boldone,r-1}_\calA\}\in\QQ.
\]
\end{dfn}

Note that $\deg_\calA(h)$ is the rational number $k$ for which $h$ is precisely on the boundary of $k\cdot\Delta^\boldone_\calA$, and that $\deg_\calA$ is subadditive on $H$:
\[
\deg_\calA(h+h')\le\deg_\calA(h)+\deg_\calA(h')\text{ for all }h,h'\in H.
\]
The $\calA$-degree defines the increasing \emph{Newton filtration} $N_\bullet=N^\calA_\bullet$ on $H$ by
\[
N_kH=\{h\in H\mid\deg_\calA(h)\le k\}\text{ for all }k\in\QQ_+.
\]

\begin{dfn}
Let $H\subseteq \ZZ^d$ be any affine semigroup with increasing $\QQ$-indexed discrete exhaustive filtration $F_\bullet$ such that $F_i+F_j\subseteq F_{i+j}$.
We let $\deg_F(h)$ denote the smallest $i$ with $h\in F_i(h)$. 
The \emph{graded associated semigroup} $\gr^F(H)$ is the set
\[
\{[h]\mid h\in H\}\sqcup \{-\infty\}
\]
subject to the rules
\[
[h]+[h']=
\begin{cases}
-\infty & \text{if $h=-\infty$ or $h'=-\infty$} \\
-\infty & \text{if $\deg_F(h+h')<\deg_F(h)+\deg_F(h')$} \\
[h+h'] & \text{otherwise.}
\end{cases}
\]
In the dictionary between semigroup operations and semigroup ring
operations, sums equal to $-\infty$ encode zerodivisors in the
associated graded ring. 
\end{dfn}

If the filtration in question is the Newton filtration to $\calA$ we
denote the associated graded semigroup by $\gr^\calA(H)$ rather than
$\gr^{N^\calA}(H)$. In contrast to filtrations induced by an additive
weight, the Newton filtration may have an associated graded object
whose generators are not immediately obvious.

\begin{exa}\label{37}
Let $H=\NN\calA$ with $\calA=\{(1,0),(2,2),(0,1)\}\subseteq \ZZ^2$, so $H=\NN(1,0)+\NN(0,1)$. 
Although $(1,0)+(0,1)$ equals $(1,1)\in H$, the corresponding sum is $-\infty$ in the associated graded semigroup $\gr^\calA_\bullet(H)$. 
To see this, note that $\deg_\calA((1,1))=1/2<2=\deg_\calA((1,0))+\deg_\calA((0,1))$. 
This also implies that $\gr^\calA_{\geq 1}(H)$ cannot contain a generating set for $\gr^\calA(H)$ and in particular the cosets of $(1,0)$, $(2,2)$ and $(0,1)$ are not generators of the graded object.
\end{exa}

The following fundamental fact is stated, but not proved, in
\cite[(6)]{Kou76}.  It puts the above example in perspective and
determines the additive structure of $\gr^\calA(H)$. In consequence,
it describes the ring structure of $\gr^\calA(\KK[H])$, which as $\ZZ
H$-graded $\KK$-vector space can be identified with $\KK[H]$.  In
particular, one finds that $\gr^\calA(\KK[H])$ contains a copy of
$\KK[H_\tau]$ for all $\tau\in\Phi^{\boldone}_\calA$, cf.~Section~\ref{CM}). 
For convenience of the reader we provide a proof.

\begin{lem}\label{3}
The equality $\deg_\calA(h+h')=\deg_\calA(h)+\deg_\calA(h')$ holds
if and only if $h,h'$ share a $\boldone$-facet cone.
\end{lem}

\begin{proof}
If $h\in H_\sigma$ with $\sigma\in\Phi^{\boldone,r-1}_\calA$  then $\deg_\calA(h)=\ell_\sigma(h)$ and hence $\ell_{\sigma'}(h)\le\ell_\sigma(h)$ for all $\sigma'\in\Phi^{\boldone,r-1}_\calA$. 
If in addition $h\notin C_{\sigma'}$ then $h/\ell_\sigma(h)\in\Delta^\boldone_\calA\smallsetminus\sigma'$ and hence $\ell_{\sigma'}(h)<\ell_\sigma(h)$.
Now let $h\in H_\sigma$, $h'\in H_{\sigma'}$, and $h+h'\in C_{\sigma''}$ for some $\sigma,\sigma',\sigma''\in\Phi^{\boldone,r-1}_\calA$. 
Then 
\begin{align*}
\deg_\calA(h+h')=\ell_{\sigma''}(h+h')
&=\ell_{\sigma''}(h)+\ell_{\sigma''}(h')\\
&\le\ell_\sigma(h)+\ell_{\sigma'}(h')=\deg_\calA(h)+\deg_\calA(h').
\end{align*}
If $h,h'\in C_\sigma''$ then one can choose $\sigma=\sigma'=\sigma''$
so that the $\calA$-degree
of $h,h'$ and $h+h'$ is evaluated by the same linear functional
$\ell_{\sigma''}$. So equality in the above display follows in this case. 
If conversely either $h\notin C_{\sigma''}$ or $h'\notin C_{\sigma''}$
then the inequality is strict by the remarks before the display.
\end{proof}

Let $\calA'\supseteq \calA$ be a second set of generators for $H$. It
defines a free presentation of the monoid $H$,
\begin{equation}\label{1}
\varphi_{\calA'}\colon F_{\calA'}=\bigoplus_{a'\in\calA'}\NN\cdot e_{a'}\onto H,\quad e_{a'}\mapsto a'.
\end{equation}
where $\{e_{a'}\mid a'\in\calA\}$ is the distinguished monoid basis of
$F_{\calA'}=\NN^{\calA'}$.  The \emph{equalizer} of $\varphi_{\calA'}$,
\[
E_{\calA'}=\{(p,q)\in F_{\calA'}\times
F_{\calA'}\mid\varphi_{\calA'}(p)=\varphi_{\calA'}(q)\},
\]
defines an equivalence relation $\sim_{\varphi_{\calA'}}$ on $F_{\calA'}$ by
$[p\sim_{\varphi_{\calA'}} q]\Leftrightarrow[(p,q)\in E_{\calA'}]$. Of
course, the set of cosets with induced additive structure is precisely
$H$.

Let $L=(L_{a'})_{a'\in\calA'}$ be the vector of Newton degrees relative to
$\calA$:
\begin{equation}\label{35}
L_{a'}=\deg_\calA(a').
\end{equation}
The corresponding linear form $\deg_L=\sum_{a'\in\calA'}L_{a'}e_{a'}^*$ on $\QQ F_{\calA'}$ defines an additive degree and an increasing filtration $L_\bullet=L^\calA_\bullet$ on $F_{\calA'}$ by
\begin{equation}\label{4}
L_kF_{\calA'}=\{p\in F_{\calA'}\mid\deg_L(p)\le k\}\text{ for all }k\in\QQ_+.
\end{equation}

The equalizer $E_{\calA'}$ is equipped with an induced $L$-filtration
via the inclusion $E_{\calA'}\subseteq F_{\calA'}\times
F_{\calA'}$. This endows $H$ with a second rational increasing
filtration besides the $\calA$-Newton filtration: the filtration given
by the degree function
\begin{equation}\label{16}
\deg_L(\varphi_{\calA'}(p))=\min\{\deg_L(q)\mid p\sim_{\varphi_{\calA'}} q\}.
\end{equation}
In order to keep the notation straight we denote this second incarnation
of $H$ with the corresponding filtration by
$\Fquot{\calA'}$.
Both the $\calA$-Newton filtration and the
filtration $L^\calA_\bullet$ are equivalent to one with index set $\NN$; in
particular, their index sets are well-ordered.

By \eqref{35}, the additivity of $\deg_L$ on $F_{\calA'}$ and
the subadditivity of $\deg_\calA$ on $H$ imply
$\deg_L(p)\ge\deg_\calA(\varphi_{\calA'}(p))$ for all $p\in
F_{\calA'}$. Thus, $\varphi_{\calA'}(L_kF_{\calA'})\subseteq N_kH$ and
\begin{equation}\label{2}
L_k(\Fquot{\calA'})\subseteq N_kH
\end{equation}
for all $k\in\QQ$ which yields a morphism of filtered semigroups
\[
\Bar\varphi_{\calA',\calA}\colon(\Fquot{\calA'},L_\bullet)\to(H,N^\calA_\bullet).
\]
To \eqref{1} corresponds the $\KK$-algebra morphism 
\begin{equation}\label{28}
\varphi^\KK_{\calA'}\colon\KK[F_{\calA'}]\onto\KK[H],\quad y^p\mapsto t^{\varphi_{\calA'}(p)}.
\end{equation}
whose kernel is the toric ideal
\begin{equation}\label{19}
I_{\calA'}=\ideal{y^p-y^q\mid (p,q)\in E_{\calA'}}.
\end{equation}
The $L$-filtration on $F_{\calA'}$ induces an $L$-filtration on
$\KK[F_{\calA'}]$ that descends to the filtration on
$\KK[F_{\calA'}]/I_{\calA'}$ given by the $L$-filtration on
$\Fquot{\calA'}$ from \eqref{16}. 
By \eqref{2} we hence have a morphism of filtered $\KK$-algebras
\begin{equation}\label{20}
\Bar\varphi^\KK_{\calA',\calA}\colon(\KK[F_{\calA'}]/I_{\calA'},L_\bullet)\to(\KK[H],N^{\calA}_\bullet).
\end{equation}
Note that 
\begin{equation}\label{40}
\gr^L(\KK[F_{\calA'}]/I_{\calA'})=\KK[F_{\calA'}]/\gr^L(I_{\calA'}).
\end{equation}
While \eqref{28} is a surjection with kernel $I_{\calA'}$,
\eqref{20} is not necessarily an epimorphism of filtered
algebras. 
In particular, it may fail to induce a surjection of associated graded algebras; 
for the most elementary example see Example~\ref{37} with $\calA=\calA'$.

The following result describes precisely which sets $\calA'$ produce an isomorphism in \eqref{20}. 
In this proposition, and hereafter, we make use of the fact that $\Delta^{\boldone,k}_{\calA}=\Delta^{L,k}_{\calA'}$ (cf.~Notation \ref{49} and \eqref{35}) and hence $\Phi^{\boldone,k}_{\calA}=\Phi^{L,k}_{\calA'}$ for all $k$.

\begin{thm}\label{34}
For $\calA'\supseteq \calA$ the following conditions are equivalent:
\begin{enumerate}
\item\label{34a} For each $\sigma\in\Phi^{\boldone,r-1}_\calA$, $\calA'_\sigma$ contains a generating set for the monoid $H_\sigma$.
\item\label{34b} For each $\tau\in\Phi^{\boldone}_\calA$, $\calA'_\tau$ contains a generating set for the monoid $H_\tau$.
\item\label{34c} The surjection \eqref{28} maps $L_k(F_{\calA'})$ onto $N_k(H)$ for all $k\in\QQ$.
\item\label{34d} The morphism \eqref{20} is an isomorphism of filtered semigroups.
\end{enumerate}
\end{thm}

\begin{proof}
It is clear that the last two conditions are equivalent. As faces are
the intersection of the facets they are contained in, the first two
conditions are equivalent as well.

Pick $h\in H$ with $\deg_\calA(h)=l$, and let
$C_\sigma$ be any cone in $\Phi^{\boldone,r-1}_\calA$ that contains $h$. 
If condition \eqref{34a} holds then there is a relation 
$h=\sum_{a'\in \calA'_\sigma}k_{a'}a'$
and by Lemma~\ref{3} we have
\[
\deg_\calA(h)=\sum_{a'\in\calA'_\sigma}k_{a'}\deg_\calA(a')=\sum_{a'\in
\calA'_\sigma}k_{a'}\deg_L(e_{a'})=\deg_L(\sum_{a'\in\calA'_\sigma}k_{a'}e_{a'}).
\]
Hence, $h\in\varphi_{\calA'}(L_l F_{\calA'})$ and \eqref{34c} follows.

Conversely, assume condition \eqref{34a} fails and let,
for suitable $\sigma\in\Phi^{\boldone,r-1}_\calA$, $h\in H_\sigma$ be an element not contained in $\NN\calA'_\sigma$. Then any
relation $h=\sum_{a'\in \calA'}k_{a'}a'$ involves at least one
$a'\in\calA'$ outside $C_\sigma$.  As $h\in C_\sigma$ but $a'\not\in
C_\sigma$, Lemma~\ref{3} asserts that $\deg_\calA(h)$ is strictly less
than $\deg_L(\sum_{a'\in\calA'}k_{a'}e_{a'})$. This being so for all
presentations for $h$, $h$ cannot be in $\varphi_{\calA'}(L_l
F_{\calA'})$ and \eqref{34c} cannot hold.
\end{proof} 

Algorithm~\ref{7} below computes the Newton filtration relative to $\calA$ by enlarging $\calA'\supseteq\calA$ until the conditions in Theorem \ref{34} are fulfilled.
The approach is to compare not the filtrations $L_\bullet$ and $N^\calA_\bullet$ or the corresponding graded objects but the defining relations of the latter and to systematically add generators to $\calA'$ to reach equality of these relations.
While Lemma \ref{3} determines the relations of $H$, the relations $\gr^L(E_{\calA'})$ have been studied in general in \cite{SW05} from where we shall extract Corollary \ref{21} for our purposes.
In Lemma \ref{42} it will turn out that the above equality of relations enforces condition \ref{34}.\eqref{34a} which will justify our procedure.

The following is a reformulation of \cite[Thm.~2.15]{SW05} which is valid for general $L$. Recall that Lemma~\ref{9} contains the special case $L=\boldzero$.

\begin{prp}\label{17}
For any $\calA$ and any $L=(L_a)_{a\in\calA}$, any $\ZZ H$-graded prime ideal in $\gr^L(\KK[F_\calA]/I_\calA)$ is of the form 
\[
I^L_\tau=I_{\{a\in\calA\mid a^L\in\tau\}}+\ideal{y_a\mid a\in\calA, a^L\not\in\tau}
\]
for some $\tau\in\Phi^L_\calA$. 
In particular, the radical of $\gr^L(I_\calA)$ is $\bigcap_{\sigma\in\Phi^{L,r-1}_\calA}I^L_\sigma$. \qed
\end{prp}

The following corollary adapts Proposition~\ref{17} to our special choice of $L$ and yields the core of our procedure.
Its first part relates relations in $\gr^L(E_{\calA'})$ to those of $H$ defined by $\boldone$-facets in Lemma \ref{3}.
Its second part serves to determine all $\boldone$-facets from a Gr\"obner basis in Steps \ref{7a} and \ref{7b} of our algorithm.

\begin{dfn}
We call an \emph{$L$-leading term of $E_{\calA'}$} an element $p\in F_{\calA'}$ such that $\deg_L(p)>\deg_L(q)$ for some $(p,q)\in E_{\calA'}$. 
\end{dfn}

\begin{cor}\label{21}
For $a'_1,a'_2\in\calA'$, the following conditions are equivalent:
\begin{enumerate}
\item\label{21a} For some $k\ge1$, $k(e_{a'_1}+e_{a'_2})$ is an
$L$-leading term of $E_{\calA'}$.
\item\label{21z} The elements $a'_1,a'_2$ do not share a
  $\boldone$-facet cone.
\end{enumerate}
If $a_1,a_2$ are actually in $\calA$, the following conditions are
equivalent:
\begin{enumerate}
\setcounter{enumi}{2}
\item\label{21d} One has $\deg_\calA(a_1)=1=\deg_\calA(a_2)$, and 
$a_1,a_2$ do not share a $\boldone$-facet cone.
\item\label{21e} For some $k\ge1$, $k(e_{a_1}+e_{a_2})$ is a
$\boldone$-leading term of $E_{\calA}$, but neither
$ke_{a_1}$ nor $ke_{a_1}$ is a
$\boldone$-leading term of $E_{\calA}$ for any $k\ge1$.
\end{enumerate}
\end{cor}

\begin{proof}
By \eqref{19} and \eqref{40}, condition \ref{21}.\eqref{21a} holds if
and only if $t^{a'_1}t^{a'_2}$ is nilpotent in
$\gr^L(\KK[F_{\calA'}]/I_{\calA'})$.  By Proposition~\ref{17} this
happens exactly when $t^{a'_1}t^{a'_2}$ is contained in
$\bigcap_{\tau\in\Phi_{\calA'}^L}I^L_\tau$.  This is in turn
equivalent to $(a'_1)^L,(a'_2)^L\in\tau$ for no
$\tau\in\Phi^L_{\calA'}=\Phi^\boldone_\calA$.  By definition, 
$\deg_\calA((a')^L)=1$ for all $a'\in\calA'$ and hence $(a')^L\in\tau$
is equivalent to $a'\in C_\tau$ for all
$\tau\in\Phi_{\calA'}^L=\Phi_\calA^\boldone$.  This proves equivalence
of \ref{21}.\eqref{21a} and \ref{21}.\eqref{21z}.

By the equivalence of conditions \ref{21}.\eqref{21a} and
\ref{21}.\eqref{21z} above, the first condition in
\ref{21}.\eqref{21e} means that $a_1,a_2\in\sigma$ for no
$\sigma\in\Phi^{\boldone,r-1}_\calA$.  On the other hand, the second
condition means that neither $t^{a_1}$ nor $t^{a_2}$ is nilpotent in
$\gr^\boldone(\KK[H])$, which in turn, by Proposition~\ref{17}, is
equivalent to $\deg_\calA(a_1)=1=\deg_\calA(a_2)$. But under this
latter condition, $a_1,a_2\in\sigma$ for no $\sigma\in
\Phi^{\boldone,r-1}_\calA$ is equivalent to $a_1,a_2\in C_\sigma$ for
no $\sigma\in\Phi^{\boldone,r-1}_\calA$. The equivalence of
\ref{21}.\eqref{21d} and \ref{21}.\eqref{21e} follows.
\end{proof}

We now state the algorithm and explain how to carry out its steps in
practice, and then prove termination and correctness.

\begin{alg}[Newton filtration on an affine semigroup]\label{7}\
\begin{enumerate}[1.]
\item[\sc{Input}:] $\calA\subseteq\ZZ^d$ such that $H=\NN\calA$ is positive.
\item[\sc{Output}:] $H\supseteq\calA'\supseteq\calA$ and
  $L=(L_a)_{a\in\calA'}$ such that \eqref{2} is an equality for all $k$.
\item\label{7a} Compute the set $\calB$ of $e_{a_1}+e_{a_2}$ where $a_1,a_2\in\calA$ such that $k(e_{a_1}+e_{a_2})$ is a $\boldone$-leading term of $E_{\calA}$ for some $k\ge1$ but neither $ke_{a_1}$ nor $ke_{a_2}$ is a $\boldone$-leading term of $E_{\calA}$ for any $k\ge1$.
\item\label{7b} Determine all $\sigma\in\Phi_\calA^{\boldone,r-1}$ by the rule: $a_1,a_2\not\in\sigma$ iff $e_{a_1}+e_{a_2}\in\calB$.
\item\label{7c} For each $\sigma\in\Phi_\calA^{\boldone,r-1}$, compute the linear form $\ell_\sigma$ that equals $1$ on $\sigma$.
\item\label{7d} For every pair $a_1,a_2\in\calA$ with
  $\deg_\calA(a_1)\deg_\calA(a_2)<1$ update $\calB$ to include
  $e_{a_1}+e_{a_2}$ if $a_1,a_2$ do not share a $\boldone$-facet cone.
\item\label{7e} Initialize $\calA'=\calA$ and
  $L=(L_{a'})_{a'\in\calA'}$ by $L_{a'}=\deg_\calA(a')$.
\item\label{7f} Remove from $\calB$ all $L$-leading terms of $E_{\calA'}$.
\item\label{7g} For each $e_{b_1}+e_{b_2}\in\calB$ do the following:
\begin{enumerate}[\ref{7g}.1]
\item\label{7g1} Update $\calA'$ to include $b_1+b_2$.
\item\label{7g2} Update $L$ with $L_{b_1+b_2}=\deg_\calA(b_1+b_2)$.
\item\label{7g3} Update $\calB$ to include $e_{a'_1}+e_{a'_2}$ whenever
  $a'_1, a'_2\in\calA'$ do not share a $\boldone$-facet cone.
\end{enumerate}
\item\label{7h} If $\calB\ne\emptyset$ continue with Step \ref{7f}.
\item\label{7i} Return $\calA'$ and $L$.
\end{enumerate}
\end{alg}

\begin{rmk}\label{39}\
\begin{asparaenum}

\item\label{39a} By the correspondence \eqref{19} between $E_{\calA'}$ and $I_{\calA'}$, the set of all $L$-leading terms of $E_{\calA'}$ is computable by Gr\"obner basis methods in rings of polynomials.
Such computations may be carried out with standard programs such as \cite{M2,GPS05}. 

\item\label{39b} By the second part of Corollary \ref{21}, Step \ref{7a} adds expressions $e_{a_1}+e_{a_2}$ to the queue $\calB$ where $a_1,a_2\in\calA$ are in different facets of the $\boldone$-umbrella of $\calA$.
The conditions in Step~\ref{7a} translate to $y_{a_1}y_{a_2}\in\sqrt{\gr^\boldone(I_\calA)}\not\ni y_{a_1},y_{a_2}$ which may be tested via {\tt decompose} in \cite{M2} or with {\tt radical} in \cite{GPS05}. 
Alternatively, in large examples, one can use the estimate in Lemma~\ref{38} below to test this membership.

\item\label{39c} The condition $a\in\sigma$ is, via Definition~\ref{36}, equivalent to $\deg_\calA(a)=\ell_\sigma(a)$. 

\item\label{39d} Step~\ref{7d} adds expressions $e_{a_1}+e_{a_2}$ to the queue $\calB$ where $a_1,a_2\in\calA$ are in different $\boldone$-facet cones and at least one of them is  in the interior of $\Delta^\boldone_\calA$.
Starting with the first passage of Step \ref{7f}, the queue $\calB$ indicates pairs $\{a'_1,a'_2\}$ for which $[a'_1]+[a'_2]=-\infty$ in $\gr^\calA(H)$ but not in $\gr^L(\Fquot{\calA'})$.

\item\label{39e} By \eqref{40}, the $L$-leading terms of an $L$-Gr\"obner basis of $I_{\calA'}$ give a presentation of $\gr^\calA(\KK[H])=\gr^L(\KK[F_{\calA'}]/I_{\calA'})$, cf.~\eqref{46} in Example \ref{15}.

\end{asparaenum}
\end{rmk}

Our first task is to assure convergence of the algorithm. 
To this end fix a facet $\sigma\in\Phi^{\boldone,r-1}_\calA$. 
Any increasing sequence of subsets $\calA_\sigma=\calA^0_\sigma\subseteq\calA^1_\sigma\subseteq\calA^2_\sigma\subseteq\cdots\subseteq H_\sigma$ leads to a stabilizing sequence of semigroups 
\begin{equation}\label{44}
\NN\calA_\sigma=\NN\calA^0_\sigma\subseteq\NN\calA^1_\sigma\subseteq\NN\calA^2_\sigma\subseteq\cdots\subseteq H_\sigma. 
\end{equation}
Namely, one obtains an increasing sequence of $\KK[\NN\calA_\sigma]$-modules
\begin{equation}\label{43}
\KK[\NN\calA_\sigma]=\KK[\NN\calA^0_\sigma]\subseteq\KK[\NN\calA^1_\sigma]\subseteq\KK[\NN\calA^2_\sigma]\subseteq\cdots\subseteq\KK[H_\sigma]. 
\end{equation}
Essentially by Gordan's lemma, $H_\sigma$ is finitely generated.
Since $\calA$ contains elements on the extremal rays of $C_\sigma$, $\KK[H_\sigma]$ is a
finite integral extension of $\KK[\calA_\sigma]$ and hence a Noetherian $\KK[\calA_\sigma]$-module. 
Thus, \eqref{43} stabilizes and hence so does \eqref{44}. 
By finiteness of $\Phi^{\boldone,r-1}_\calA$, it follows that eventually any new element $b_1+b_2$ of $\calA'$ suggested by Step~\ref{7g} is already in $\NN\calA'_\sigma$ for some $\sigma\in\Phi^{\boldone,r-1}_\calA$.

Suppose the algorithm has reached this stage and let $e_{b_1}+e_{b_2}\in\calB$.
By Remark~\ref{39}.\eqref{39b} and Step~\ref{7d}, $b_1,b_2$ do not share a $\boldone$-facet cone.
By Lemma~\ref{3}, with $a''=b_1+b_2$,
\[
\deg_L(e_{b_1}+e_{b_2})=\deg_L(e_{b_1})+\deg_L(e_{b_2})=\deg_\calA(b_1)+\deg_\calA(b_2)>\deg_\calA(a'').
\]
By the stability hypothesis on \eqref{44}, $a''\in\NN\calA'_\sigma$ for some $\sigma\in\Phi^{\boldone,r-1}_\calA$ and one can write $a''=\sum_{a'\in\calA'_\sigma}k_{a'}a'$ where
\[
\deg_\calA(e_{a''})=\sum_{a'\in\calA'_\sigma}k_{a'}\deg_\calA(e_{a'})=\sum_{a'\in\calA'_\sigma}k_{a'}\deg_L(e_{a'})=\deg_L(e_{a''})
\]
by Lemma~\ref{3} again. 
Thus, $(e_{b_1}+e_{b_2},\sum_{a'\in\calA'_\sigma}k_{a'}e_{a'})\in E_{\calA'}$ with $L$-leading term $e_{b_1}+e_{b_2}$ in contradiction to Step~\ref{7f}.
We conclude that $\calB=0$ and the algorithm terminates.

\medskip

We now prove that the algorithm computes what we want. So we assume
that $\calB$ has been reduced to the empty set and we let from now on
$\calA'$ denote the stable value of the generating set for $H$ that
forms the output of Algorithm~\ref{7}.

\begin{lem}\label{42}
If $\calB=\emptyset$ then for $a'_1,a'_2\in\calA'$ the following
are equivalent:
\begin{enumerate}
\item\label{42b} $e_{a'_1}+e_{a'_2}$ is an $L$-leading term of $E_{\calA'}$.
\item\label{42a} $a'_1,a'_2$ do not share a $\boldone$-facet cone.
\end{enumerate}
In particular, $\NN\calA'_\sigma=H_\sigma$ for all $\sigma\in\Phi^{\boldone,r-1}_\calA$.
\end{lem}

\begin{proof}
In order to show the announced equivalence it suffices by
Corollary~\ref{21} to show that condition \ref{21}.\eqref{21a} implies
condition \ref{42}.\eqref{42b}.  So assume that $k(e_{a'_1}+e_{a'_2})$
is an $L$-leading term of $E_{\calA'}$, $k$ being minimal in that
respect. By Corollary~\ref{21}, $a'_1$ and $a'_2$ do not share a
$\boldone$-facet cone and hence $e_{a'_1}+e_{a'_2}$ was added to
$\calB$ at some point in the algorithm. 
This is obvious from Step~\ref{7g} if not both $a'_1,a'_2$ are in $\calA$ and from \eqref{39b} and \eqref{39d} in Remark~\ref{39} otherwise.
As $\calB=\emptyset$, $e_{a'_1}+e_{a'_2}$ was eliminated from the queue in Step~\ref{7f} and hence must be an $L$-leading term of $E_{\calA'}$ itself.

To prove the second claim, let $a\in H_\sigma$.  Since $\calA'$
generates $H$, one can write $a=\sum_{a'\in\calA'}k_{a'}a'$.  Of all
such expressions pick one for which
$\deg_L(\sum_{a'\in\calA'}k_{a'}e_{a'})$ is minimized.  Suppose
$a'_1,a'_2$ make a contribution to the sum and \ref{42}.\eqref{42a}
and hence \ref{42}.\eqref{42b} holds.  Then $a'_1+a'_2=\sum_{a''\in
\calA'}k_{a''}a''$ with
$\deg_L(e_{a'_1}+e_{a'_2})>\deg_L(\sum_{a''\in\calA'}k_{a''}e_{a''})$.
So $\sum_{a'\in\calA'}k_{a'}e_{a'}-e_{a'_1}-e_{a'_2}+\sum_{a''\in
\calA'}k_{a''}e_{a'}$ maps to $a$ under $\varphi_{\calA'}$ in
display \eqref{1} but has smaller $L$-degree than
$\sum_{a'\in\calA'}k_{a'}e_{a'}$.  By contradiction, there must be a
$\sigma\in\Phi^{\boldone,r-1}_\calA$ for which $C_\sigma$ contains all
terms in $\sum_{a'\in\calA'}k_{a'}a'$ and so $a\in\NN\calA'_\sigma$ as
required.
\end{proof}

Corollary~\ref{21}, Theorem~\ref{34}, Lemma~\ref{42}, and the
arguments after Remark~\ref{39} combine to the main theorem of this section.

\begin{thm}
Algorithm \ref{7} terminates and is correct.\qed
\end{thm}

The following bound for the torsion order $k$ in Corollary \ref{21} can be useful in practice, cf.\ Remark~\ref{39}.\eqref{39b}.

\begin{dfn}\label{45}
We denote by $M(\calA)$ be the largest absolute value of a maximal minor of a matrix $A$ whose columns are a $\ZZ H$ basis representation of the elements of $\calA$.
\end{dfn}

\begin{lem}\label{38}
For any $p\in F_\calA$, if $kp$ is an $L$-leading term of $E_{\calA}$ for some $k\ge1$ then also for $k=M(\calA)$.
\end{lem}

\begin{proof}
Proposition~\ref{17} spells out when $t^p$ is nilpotent in
$\gr^L(\KK[H])$.
We may reduce to the cases $p=e_a$ where $a^L\not\in\sigma$ for any
$\sigma\in\Phi_\calA^{L,r-1}$, and $p=e_{a_1}+e_{a_2}$ where
$a_1^L,a_2^L\in\sigma$ for no $\sigma\in\Phi_\calA^{L,r-1}$.

In the first case, $a^L$ lies in the interior of $\Delta^L_\calA$ and
we choose $\sigma\in\Phi_\calA^{L,r-1}$ with $a\in C_\sigma$. 
Then there is a subset $\calB\subseteq\calA$ of $r$ linearly independent elements such
that $\calB^L=\{b^L\mid b\in\calB\}\subseteq\sigma$ and
$a\in\QQ_+\calB$. Let $B$ be a matrix whose columns are the elements
of $\calB$, denote the adjoint matrix of $B$ by $\check B$, and let
$k=\det(B)\ne 0$.  Then $ka=B\check Ba$ implies that
$ka=\sum_{b\in\calB}k_bb$ for some $k_b\in\ZZ$. Since 
$ka\in\QQ_+\calB$ and since 
$\calB$ is linearly independent it follows that each $k_b\in\NN$. Rewriting this as
$kL_aa^L=\sum_{b\in\calB}k_bL_bb^L$, the above conditions on $a^L$
imply that
\[
\deg_L(ke_a)=kL_a>\sum_{b\in\calB}k_bL_b=\deg_L(\sum_{b\in\calB}k_be_b).
\]
Thus, $\left(ke_a,\sum_{b\in\calB}k_be_b\right)\in E_\calA$ and $ke_a$
is an $L$-leading term of $E_\calA$.

The proof of the second case is analogous: the convex combination
$\frac{L_1a_1^L+L_2a_2^L}{L_1+L_2}=\frac{a_1+a_2}{L_1+L_2}$ of $a_1^L$
and $a_2^L$ lies in the interior of $\Delta^L_\calA$.  As above,
choose a linearly independent subset $\calB\subseteq\calA$ such that
$a_1+a_2\in\QQ_+\calB$.  This time, an equality
\[
(L_1+L_2)k\frac{L_1a_1^L+L_2a_2^L}{L_1+L_2}=k(a_1+a_2)=\sum_{b\in\calB}k_bb=\sum_{b\in\calB}k_bL_bb^L,
\]
with $k_b\in\NN$, 
implies that $\deg_L(k(e_{a_1}+e_{a_2}))>\deg_L(\sum_{b\in\calB}k_be_b)$
and hence that $k(e_{a_1}+e_{a_2})$ is an $L$-leading term of $E_\calA$.

By Definition~\ref{45}, $M(\calA)\ge\det(B)=k$ which completes the proof.
\end{proof}

\begin{exa}\label{15}
Consider the affine semigroup $H$ generated by 
\[
\calA=\{a_1,\dots,a_6\}=
\left\{
\begin{pmatrix}2\\0\\0\end{pmatrix},
\begin{pmatrix}3\\0\\0\end{pmatrix},
\begin{pmatrix}0\\1\\0\end{pmatrix},
\begin{pmatrix}1\\1\\0\end{pmatrix},
\begin{pmatrix}2\\0\\1\end{pmatrix},
\begin{pmatrix}0\\2\\1\end{pmatrix}
\right\}.
\]
\begin{figure}[htbp]
\setlength{\unitlength}{0.8mm}
\begin{picture}(100,60)(0,0){
\put(50,0){
\put(0,0){\vector(0,1){50}}
\put(0,0){\vector(1,0){50}}
\put(0,0){\vector(-1,1){50}}
\multiput(0,0)(0,20){3}{\path(1,0)(-1,0)}
\multiput(0,0)(20,0){3}{\path(0,1)(0,-1)}
\multiput(-14,14)(-14,14){3}{\path(-1,-1)(1,1)}
\put(0,-2){\makebox(0,0)[t]{$0$}}
\put(2,20){\makebox(0,0)[l]{$1$}}
\put(2,40){\makebox(0,0)[l]{$2$}}
\put(1,49){\makebox(0,0)[l]{$t_3$}}
\put(20,-2){\makebox(0,0)[t]{$1$}}
\put(40,-2){\makebox(0,0)[t]{$2$}}
\put(49,-1){\makebox(0,0)[t]{$t_2$}}
\put(-15,13){\makebox(0,0)[tr]{$1$}}
\put(-29,27){\makebox(0,0)[tr]{$2$}}
\put(-43,41){\makebox(0,0)[tr]{$3$}}
\put(-49,48){\makebox(0,0)[tr]{$t_1$}}
\put(0,0){\texture{
bbbbbbbb bbbbbbbb bbbbbbbb bbbbbbbb bbbbbbbb bbbbbbbb bbbbbbbb bbbbbbbb 
bbbbbbbb bbbbbbbb bbbbbbbb bbbbbbbb bbbbbbbb bbbbbbbb bbbbbbbb bbbbbbbb 
bbbbbbbb bbbbbbbb bbbbbbbb bbbbbbbb bbbbbbbb bbbbbbbb bbbbbbbb bbbbbbbb 
bbbbbbbb bbbbbbbb bbbbbbbb bbbbbbbb bbbbbbbb bbbbbbbb bbbbbbbb bbbbbbbb}
\shade\path(20,0)(40,20)(6,14)(20,0)\thicklines\path(20,0)(40,20)(6,14)(20,0)}
\put(0,0){\texture{
88888888 88888888 88888888 88888888 88888888 88888888 88888888 88888888 
88888888 88888888 88888888 88888888 88888888 88888888 88888888 88888888 
88888888 88888888 88888888 88888888 88888888 88888888 88888888 88888888 
88888888 88888888 88888888 88888888 88888888 88888888 88888888 88888888}
\shade\path(40,20)(-28,48)(-42,42)(40,20)\thicklines\path(40,20)(-28,48)(-42,42)(40,20)}
\put(0,0){\texture{
99999999 99999999 99999999 99999999 99999999 99999999 99999999 99999999 
99999999 99999999 99999999 99999999 99999999 99999999 99999999 99999999 
99999999 99999999 99999999 99999999 99999999 99999999 99999999 99999999 
99999999 99999999 99999999 99999999 99999999 99999999 99999999 99999999}
\shade\path(6,14)(40,20)(-42,42)(6,14)
\thicklines\path(6,14)(40,20)(-42,42)(6,14)}
\put(-28,28){\circle*{1}\put(2,0){\makebox(0,0)[l]{$a_1$}}}
\put(-42,42){\circle*{1}\put(0,2){\makebox(0,0)[b]{$a_2$}}}
\put(20,0){\circle*{1}\put(2,0){\makebox(0,0)[bl]{$a_3$}}}
\put(6,14){\circle*{1}\put(0,-1){\makebox(0,0)[tr]{$a_4$}}}
\put(-28,48){\circle*{1}\put(0,1){\makebox(0,0)[b
]{$a_5$}}}
\put(40,20){\circle*{1}\put(1,0){\makebox(0,0)[l]{$a_6$}}}
\put(-8,28){\circle*{1}\put(1,0){\makebox(0,0)[l]{$a'_7$}}}
\put(-8,48){\circle*{1}\put(1,0){\makebox(0,0)[l]{$a'_8$}}}
\put(-22,62){\circle*{1}\put(1,0){\makebox(0,0)[l]{$a'_9$}}}
\dottedline{1}(0,0)(-8,28)
\dottedline{1}(0,0)(-3.23,19.38)
\dottedline{1}(-6,36)(-8,48)
\dottedline{1}(0,0)(-7.83,22.07)
\dottedline{1}(-13.2,37.2)(-22,62)
\put(-6,21){\circle*{1}}
\put(-6,36){\circle*{1}}
\put(-13.2,37.2){\circle*{1}}
\put(6,22){\makebox(0,0)[b]{$\sigma_1$}}
\put(-26,41){\makebox(0,0)[b]{$\sigma_2$}}
\put(20,8){\makebox(0,0)[b]{$\sigma_3$}}
}
}
\end{picture}
\caption{$\calA'$ as generated by Algorithm~\ref{7} applied to Example~\ref{15} and its central projection onto the $(\calA,\boldone)$-umbrella}\label{47}
\end{figure}

We follow the steps in Algorithm~\ref{7}.
Using {\sc Singular} \cite{GPS05}, we compute a $\boldone$-Gr\"obner basis
\begin{gather*}
y_{2}y_{3}-y_{1}y_{4}, \quad
y_{4}^{2}y_{5}-y_{1}^{2}y_{6}, \quad
y_{3}y_{4}y_{5}-y_{2}y_{6}, \quad
y_{3}^{2}y_{5}-y_{1}y_{6}, \\
y_{1}y_{3}^{2}-y_{4}^{2}, \quad
y_{1}^{2}y_{3}-y_{2}y_{4}, \quad
y_{1}^{3}-y_{2}^{2}
\end{gather*}
of $I_\calA$ and thus a Gr\"obner basis 
\begin{gather}\label{50}
y_{2}y_{3}-y_{1}y_{4}, \quad
y_{4}^{2}y_{5}-y_{1}^{2}y_{6}, \quad
y_{3}y_{4}y_{5}, \quad
y_{3}^{2}y_{5}, \quad
y_{1}y_{3}^{2}, \quad
y_{1}^{2}y_{3}, \quad
y_{1}^{3} 
\end{gather}
of $\gr^\boldone(I_\calA)$.  
Obviously, $e_{a_3}+e_{a_5}\in\calB$ after Step~\ref{7a} while $e_{a_1}+e_{a_3}$ is not added to $\calB$ since $y^3_3$ is a $\boldone$-leading term. 
To complete Step~\ref{7a}, compute $M(\calA)=6$ and reduce all $6$-th powers of $y_iy_j$ and $y_k$ with respect to the Gr\"obner basis \eqref{50} to find
\[
\calB=\{e_{a_2}+e_{a_3},e_{a_3}+e_{a_5},e_{a_4}+e_{a_5}\}.
\]
By the rule in Step~\ref{7b}, 
\[
\Phi_\calA^{\boldone,r-1}=\{\sigma_1,\sigma_2,\sigma_3\}=\left\{\{a_2,a_4,a_6\},\ \{a_2,a_5,a_6\},\ \{a_3,a_4,a_6\}\right\}
\]
and the corresponding linear forms from Step~\ref{7c} are
\[
\ell_1=\left(\frac{1}{3},\frac{2}{3},-\frac{1}{3}\right),\quad
\ell_2=\left(\frac{1}{3},\frac{1}{3},\frac{1}{3}\right),\quad
\ell_3=(0,1,-1).
\]
Note that $2a_2=3a_1$, and hence $a_1\in C_{\sigma_1}\cap C_{\sigma_2}$.
According to Step~\ref{7d}, we update 
\[
\calB=\{e_{a_1}+e_{a_3},e_{a_2}+e_{a_3},e_{a_3}+e_{a_5},e_{a_4}+e_{a_5}\}.
\]
In Step~\ref{7e}, we compute $\calA$-degrees of $\calA'=\calA$ to obtain the weight vector 
\[
L=\left(\frac{2}{3},1,1,1,1,1\right).
\]
Using {\sc Singular} \cite{GPS05}, we compute an $L$-Gr\"obner basis
\begin{gather*}
y_{2}y_{3}-y_{1}y_{4},\quad 
y_{1}^{3}-y_{2}^{2},\quad 
y_{1}^{2}y_{3}-y_{2}y_{4},\quad 
y_{1}y_{3}^{2}-y_{4}^{2},\\
y_{4}^{2}y_{5}-y_{1}^{2}y_{6},\quad 
y_{3}y_{4}y_{5}-y_{2}y_{6},\quad  
y_{3}^{2}y_{5}-y_{1}y_{6}
\end{gather*}
of $I_{\calA'}$ and thus a Gr\"obner basis
\[
y_{2}y_{3},\quad 
y_{1}^{3}-y_{2}^{2},\quad 
y_{1}^{2}y_{3},\quad 
y_{1}y_{3}^{2},\quad 
y_{4}^{2}y_{5},\quad 
y_{3}y_{4}y_{5},\quad 
y_{3}^{2}y_{5}
\]
of $\gr^L(I_{\calA'})$ that reduces $\calB$ in Step~\ref{7f} to 
\[
\calB=\{e_{a_1}+e_{a_3},e_{a_3}+e_{a_5},e_{a_4}+e_{a_5}\}.
\]
Following Step~\ref{7g}, we set
\[
a'_7=a_1+a_3=\begin{pmatrix}2\\1\\0\end{pmatrix},\quad
a'_8=a_3+a_5=\begin{pmatrix}2\\1\\1\end{pmatrix},\quad
a'_9=a_4+a_5=\begin{pmatrix}3\\1\\1\end{pmatrix}
\]
and update
\begin{align*}
\calA'&=\calA\cup
\left\{
a'_7,
a'_8,
a'_9
\right\},\\
L&=\left(\frac{2}{3},1,1,1,1,1,1,\frac{4}{3},\frac{4}{3},\frac{5}{3}\right),\\
\calB&=\{e_{a_1}+e_{a_3},e_{a_3}+e_{a_5},e_{a_3}+e_{a_7},e_{a_3}+e_{a_8},e_{a_3}+e_{a_9},\\
&\phantom{=\{}e_{a_4}+e_{a_5},e_{a_4}+e_{a_8},e_{a_4}+e_{a_9},e_{a_5}+e_{a_7},e_{a_7}+e_{e_8},e_{a_7}+e_{a_9}\}. 
\end{align*}
The next Gr\"obner basis
\begin{gather}\label{46}
y_{1}y_{3},\ 
y_{4}y_{5},\ 
y_{3}y_{5},\ 
y_{2}y_{4}-y_{1}y_{7},\ 
y_{2}y_{3},\ 
y_{1}^{3}-y_{2}^{2},\ 
y_{4}y_{8},\ 
y_{3}y_{8},\ 
y_{2}y_{8}-y_{1}y_{9},\\
y_{5}y_{7},\ 
y_{3}y_{7},\ 
y_{1}^{2}y_{4}-y_{2}y_{7},\ 
y_{4}y_{9},\ 
y_{3}y_{9},\ 
y_{7}y_{8},\ 
y_{1}^{2}y_{8}-y_{2}y_{9},\ 
y_{1}y_{5}y_{6}-y_{8}^{2},\nonumber\\
y_{1}y_{4}^{2}-y_{7}^{2},\ 
y_{7}y_{9},\ 
y_{2}y_{5}y_{6}-y_{8}y_{9},\ 
y_{1}y_{8}^{2}-y_{9}^{2},\ 
y_{8}^{4}-y_{5}y_{6}y_{9}^{2}\nonumber
\end{gather}
of $\gr^L(I_{\calA'})$ in Step~\ref{7f} reduces $\calB$ to zero and the procedure terminates.
The relations \eqref{46} define a presentation of $\gr^\calA(\KK[H])$ as quotient of $\KK[y_1,\dots,y_9]$.
\end{exa}

\section{Cohen--Macaulayness of the Newton graded toric ring}\label{CM}

In this section we investigate Cohen--Macaulayness of the Newton
graded ring $\gr^\calA(\KK[H])$.  For $1\le k\le r$, let
\[
A_k=A_k(H)=\bigoplus_{\tau\in\dotPhi^{\boldone,k-1}_\calA}\KK[H_\tau];
\]
in particular,
\[
A_r=\bigoplus_{\sigma\in\Phi^{\boldone,r-1}_\calA}\KK[H_\sigma].
\]
By Lemmas \ref{3} and \ref{9}, for $\tau\in\Phi^\boldone_\calA$, $\KK[H\smallsetminus C_\tau]$ is an ideal in $\gr^\calA(\KK[H])$ and therefore
\[
\gr^\calA(\KK[H])\onto\gr^\calA(\KK[H])/\KK[H\smallsetminus C_\tau]=\KK[H_\tau]\subseteq\gr^\calA(\KK[H]) 
\]
are maps of $\KK$-algebras.
In particular, for $\sigma\in\Phi^{\boldone,r-1}_\calA$,
\begin{equation}\label{13}
\frakm\cdot\KK[H_\sigma]=\frakm_\sigma
\end{equation}
and hence, for $1\le k\le r$,
\[
H^i_\frakm(A_k)=\bigoplus_{\tau\in\dotPhi^{\boldone,k-1}_\calA}H^i_{\frakm_\tau}\KK[H_\tau]
\]\
for all $i$.
By \cite[Prop.~2.6]{Kou76}, there is an exact sequence
\begin{equation}\label{12}
0\to\gr^\calA(\KK[H])\to A_r\to A_{r-1}\to\cdots\to A_1\to 0
\end{equation}
in the category of $\gr^\calA(\KK[H])$-modules. This sequence expresses
$\gr^\calA(\KK[H])$ as the product of the rings $\KK[H_\sigma]$ over
facets $\sigma\in\Phi^{\boldone,r-1}_\calA$ modulo the identification
of border rings.

Applying the local cohomology functor $R\Gamma_\frakm$ to \eqref{12}, one obtains a convergent
spectral sequence
\begin{equation}\label{11}
E_1^{p,q}=H^q_\frakm(A_{r-p})\Longrightarrow H^{r-p+q}_\frakm(\gr^\calA(\KK[H])).
\end{equation}
It follows that if each $\KK[H_\tau]$, $\tau\in\dotPhi^\boldone_\calA$, is Cohen--Macaulay then the spectral sequence collapses on the $E_1$-page and $\gr^\calA(\KK[H])$ must also be Cohen--Macaulay as was already shown in \cite{Kou76} with a slightly different argument. 

The local cohomology approach facilitates the search for examples where $\KK[H]$ is Cohen--Macaulay but where
$\gr^\calA(\KK[H])$ is not Cohen--Macaulay and where hence not all of the $\KK[H_\tau]$ can be Cohen--Macaulay themselves.
First, observe that the multi-degrees of $A_k$ and hence those of $R\Gamma_\frakm(A_k)$ are contained in
\[
Q_k=\bigcup_{\tau\in\dotPhi^{\boldone,k-1}_\calA}\QQ\tau.
\]
Next, suppose that $\KK[H]$ is of dimension $r=2$.  Then the spectral
sequence degenerates into a long exact sequence of $\ZZ H$-graded
modules which, since each $\KK[H_\tau]$ is a domain, starts with
\[
0\to H^1_\frakm(\gr^\calA(\KK[H]))\to H^1_\frakm(A_2)\to H^1_\frakm(A_1)\to\cdots
\]
Thus, if one can find a Cohen--Macaulay ring $\KK[H]$ where the
multi-degrees of $H^1_\frakm(A_2)$ are not contained in $Q_1$ then
$H^1_\frakm(\gr^\calA(\KK[H]))$ cannot be zero and therefore
$\gr^\calA(\KK[H])$ not Cohen--Macaulay.  

In order to arrive at such example one can start with a two-dimensional cone $C_H$ of a semigroup $H$ where $\KK[H]$ is not Cohen--Macaulay.
Then one strategically adds new elements outside $C_H$ to $H$ to make $H=\Bar H$, cf.~Corollary~\ref{14}. 
We illustrate the idea by an example.

\begin{exa}\label{51}
Let $H$ be the affine semigroup generated by 
\[
\calA=\{a_1,\dots,a_4\}=
\left\{
\begin{pmatrix}1\\0\end{pmatrix},
\begin{pmatrix}1\\1\end{pmatrix},
\begin{pmatrix}1\\3\end{pmatrix},
\begin{pmatrix}1\\4\end{pmatrix}
\right\}
\]
so that $\KK[H]$ is the coordinate ring of the pinched rational quartic. 
As is well-known, $H^1_\frakm(\KK[H])$ is the one-dimensional vector
space spanned by the class of the $1$-cocycle $(y_2^2/y_1,y_3^2/y_4)$
in the \v Cech complex to $y_1$ and $y_4$, cf.~Example~\ref{23}. 
The multi-degree of this generator equals $(1,2)$ and
$\{(1,2)\}=\Bar H\smallsetminus H$.
\begin{figure}[ht]
\caption{The $(\calA,\boldone)$-polyhedron for Example~\ref{51}}
\begin{center}
\setlength{\unitlength}{0.65mm}
\begin{picture}(170,55)(0,-5)
\put(0,0){\texture{
88888888 88888888 88888888 88888888 88888888 88888888 88888888 88888888 
88888888 88888888 88888888 88888888 88888888 88888888 88888888 88888888 
88888888 88888888 88888888 88888888 88888888 88888888 88888888 88888888 
88888888 88888888 88888888 88888888 88888888 88888888 88888888 88888888}
\shade\path(0,41)(0,0)(164,41)
\texture{
88888888 88888888 88888888 88888888 88888888 88888888 88888888 88888888 
88888888 88888888 88888888 88888888 88888888 88888888 88888888 88888888 
88888888 88888888 88888888 88888888 88888888 88888888 88888888 88888888 
88888888 88888888 88888888 88888888 88888888 88888888 88888888 88888888}
\shade\path(164,41)(0,0)(164,0)
\texture{
99999999 99999999 99999999 99999999 99999999 99999999 99999999 99999999 
99999999 99999999 99999999 99999999 99999999 99999999 99999999 99999999 
99999999 99999999 99999999 99999999 99999999 99999999 99999999 99999999 
99999999 99999999 99999999 99999999 99999999 99999999 99999999 99999999}
\shade\path(0,0)(0,20)(80,20)(60,0)(0,0)
\path(0,0)(80,20)
\dottedline{1}(160,40)(170,42.5)
\thicklines\path(0,20)(80,20)(60,0)}
\put(0,0){\vector(0,1){50}}
\put(0,0){\vector(1,0){170}}
\multiput(0,0)(0,20){3}{\path(1,0)(-1,0)}
\multiput(0,0)(20,0){9}{\path(0,1)(0,-1)}
\put(0,20){\put(1,1){\makebox(0,0)[bl]{$a_1$}}}
\put(20,20){\put(0,1){\makebox(0,0)[b]{$a_2$}}}
\put(40,20){\put(0,1){\makebox(0,0)[b]{$(1,2)$}}}
\put(60,20){\circle*{1}\put(0,1){\makebox(0,0)[b]{$a_3$}}}
\put(80,20){\circle*{1}\put(0,1){\makebox(0,0)[b]{$a_4$}}}
\multiput(0,20)(60,0){3}{
\put(0,0){\circle*{1}}
\put(20,0){\circle*{1}}
\put(40,0){\filltype{white}\circle*{1}\circle{1}}
}
\multiput(0,40)(20,0){9}{\circle*{1}}
\put(60,0){\circle*{1}\put(0,-2){\makebox(0,0)[t]{$a'=(0,3)$}}}
\put(120,0){\circle*{1}}
\put(-1,-1){\makebox(0,0)[tr]{$0$}}
\put(20,-2){\makebox(0,0)[t]{$1$}}
\put(40,-2){\makebox(0,0)[t]{$2$}}
\put(80,-2){\makebox(0,0)[t]{$4$}}
\put(100,-2){\makebox(0,0)[tr]{$5$}}
\put(120,-2){\makebox(0,0)[t]{$6$}}
\put(140,-2){\makebox(0,0)[t]{$7$}}
\put(160,-2){\makebox(0,0)[t]{$8$}}
\put(-2,20){\makebox(0,0)[r]{$1$}}
\put(169,-1){\makebox(0,0)[t]{$t_2$}}
\put(1,49){\makebox(0,0)[l]{$t_1$}}
\put(70,42){\makebox(0,0)[b]{$C_{H}$}}
\end{picture}
\end{center}
\end{figure}

We are interested in a new element $a'\in\ZZ^2\smallsetminus C_H$ such that the toric ring $\KK[H']$ of $H'=H+\NN a'$ becomes Cohen--Macaulay. 
Note that it will be a domain regardless of the choice of $a'$.
By Proposition \ref{14}.\eqref{14d}, $a'$ has to be chosen such that $(1,2)+\NN a'$ does not meet $H'$. 
If, for instance, $a'=(0,k)$ with $k\ge3$ then pairs of elements in $C_H\cap\ZZ^2$ that differ by a multiple of $a'$ are either both inside or both outside of $H$.
With such $a'$ we have therefore $(1,2)\not\in H'-\NN a'$ and hence $(1,2)\not\in\Bar H'$.
One verifies that $\Bar H'=H'$, or equivalently
\[
H^1_{\frakm_{H'}}(\KK[H'])=(\KK[H'-\NN a_1]\cap\KK[H'-\NN a'])/\KK[H']=0,
\]
and hence $\KK[H']$ is Cohen--Macaulay. 
However the summand $\KK[H]$ of $A_2(H')$ is not Cohen--Macaulay.
\end{exa}

It is natural to wonder whether the spectral sequence \eqref{11}
always degenerates at the $E_1$-page.  This would be tantamount to the
equivalence
\[
[\forall\tau\in\dotPhi^\boldone_\calA:\KK[H_\tau]\text{ is Cohen--Macaulay}]
\quad\Leftrightarrow \quad
[\gr^\calA(\KK[H])\text{ is Cohen--Macaulay}].
\]
The purpose of the following example is to show that this is not the case.

\begin{exa}\label{52}
Let $H$ be the affine semigroup generated by 
\[
\calA=\{a_1,\dots,a_6\}=
\left\{
\begin{pmatrix}0\\2\end{pmatrix},
\begin{pmatrix}0\\3\end{pmatrix},
\begin{pmatrix}-1\\1\end{pmatrix},
\begin{pmatrix}-1\\2\end{pmatrix},
\begin{pmatrix}-2\\2\end{pmatrix},
\begin{pmatrix}1\\1\end{pmatrix}
\right\}.
\] 
There are two facet cones $C_\sigma=\QQ_+a_1+\QQ_+a_3$ and $C_{\sigma'}=\QQ_+a_1+\QQ_+a_6$ where
\[
\KK[H_{\sigma'}]\cong\KK[y_1,y_2,y_6]/\ideal{y_1^3-y_2^2}
\]
is a complete intersection.
\begin{figure}[ht]
\caption{The $(\calA,\boldone)$-polyhedron for Example~\ref{52}}
\begin{center}
\setlength{\unitlength}{0.65mm}
\begin{picture}(80,80)(0,0)
\put(50,0){\texture{
88888888 88888888 88888888 88888888 88888888 88888888 88888888 88888888 
88888888 88888888 88888888 88888888 88888888 88888888 88888888 88888888 
88888888 88888888 88888888 88888888 88888888 88888888 88888888 88888888 
88888888 88888888 88888888 88888888 88888888 88888888 88888888 88888888}
\shade\path(-47,47)(0,0)(0,70)\shade\path(23,23)(0,0)(0,70)
\dottedline{1}(-47,47)(-51,51)
\dottedline{1}(0,70)(0,75)
\dottedline{1}(23,23)(27,27)
\texture{
99999999 99999999 99999999 99999999 99999999 99999999 99999999 99999999 
99999999 99999999 99999999 99999999 99999999 99999999 99999999 99999999 
99999999 99999999 99999999 99999999 99999999 99999999 99999999 99999999 
99999999 99999999 99999999 99999999 99999999 99999999 99999999 99999999}
\shade\path(0,0)(-40,40)(0,60)(20,20)(0,0)
\thicklines\path(-40,40)(0,60)(20,20)}
\put(0,0){\vector(1,0){80}}
\multiput(10,0)(20,0){4}{\path(0,1)(0,-1)}
\put(50,0){
\put(0,0){\vector(0,1){70}}
\multiput(0,0)(0,20){4}{\path(1,0)(-1,0)}
}
\put(50,0){
\put(0,40){\circle*{1}\put(-2,0){\makebox(0,0)[r]{$a_1$}}}
\put(0,60){\circle*{1}\put(-1,1){\makebox(0,0)[br]{$a_2$}}}
\put(-20,20){\circle*{1}\put(-1,-1){\makebox(0,0)[tr]{$a_3$}}}
\put(-20,40){\circle*{1}\put(0,2){\makebox(0,0)[b]{$a_4$}}}
\put(-40,40){\circle*{1}\put(0,2){\makebox(0,0)[b]{$a_5$}}}
\put(20,20){\circle*{1}\put(2,0){\makebox(0,0)[l]{$a_6$}}}}
\put(50,0){
\put(-40,-2){\makebox(0,0)[t]{$-2$}}
\put(-20,-2){\makebox(0,0)[t]{$-1$}}
\put(0,-2){\makebox(0,0)[t]{$0$}}
\put(20,-2){\makebox(0,0)[t]{$1$}}
\put(2,20){\makebox(0,0)[l]{$1$}}
\put(2,40){\makebox(0,0)[l]{$2$}}
\put(2,60){\makebox(0,0)[l]{$3$}}
\put(29,-1){\makebox(0,0)[t]{$t_1$}}
\put(1,69){\makebox(0,0)[l]{$t_2$}}
\put(10,41){\makebox(0,0)[bl]{$\sigma'$}}
\put(-21,50){\makebox(0,0)[br]{$\sigma$}}}
\end{picture}
\end{center}
\end{figure}
On the other hand, the coordinate transformation $(k_1,k_2)\mapsto(k_1,k_1+k_2)$ in $\ZZ^2$ reveals that 
\[
\KK[H_\sigma]\cong\KK[y_1,y_2,y_3,y_4]/\ideal{y_3^2y_1-y_4^2,y_3^3y_2-y_4^3,y_1^3-y_2^2,y_2y_3-y_4y_1}
\]
which is not Cohen--Macaulay since $H^1_{\frakm_\sigma}(\KK[H_\sigma])$ is the one-dimensional $\KK$-vector space generated by the \v Cech cocycle
\[
(y_2/y_1)\oplus(y_4/y_3)\in\KK[H_\sigma-\NN a_1]\oplus\KK[H_\sigma-\NN a_3].
\]
Let $\tau=\sigma\cap\sigma'$, then the natural map $H^1_{\frakm_\sigma}(\KK[H_\sigma])\to H^1_{\frakm_\tau}(\KK[H_\tau])$ induced by the map $\KK[H_\sigma]\onto\KK[H_\tau]$ of $\KK[H_\sigma]$-modules is injective since
\[
H^1_{\frakm_\tau}(\KK[H_\tau])=\KK[H_\tau-\NN a_1]/\KK[H_\tau]\cong \KK\cdot(y_2/y_1)\oplus\bigoplus_{k\ge1}\KK\cdot(y_2/y_1^2)^k
\] 
as one immediately verifies.
It follows that in the spectral sequence \eqref{11} the only potentially non zero differential
$d_1^{0,1}\colon E^{0,1}_1\to E^{1,1}_1$ is in fact injective. 
Thus the $E_2$-page has only terms on the diagonal $p+q=2$ and it follows that $\gr^\calA(\KK[H])$ is Cohen--Macaulay while the summand $\KK[H_\sigma]$ of $A_2(H)$ is not.
\end{exa}

\begin{rmk}
Using the commands {\tt dim} and {\tt depth} in {\sc Singular} \cite{GPS05}, one can verify that $\KK[H]$ in Example \ref{15} is Cohen--Macaulay while $\gr^\calA(\KK[H])$ is not.
By Proposition \ref{14}, $\KK[H_{\sigma_2}]$ is not Cohen--Macaulay because $a_2-a_1=(1,0,0)\in\Bar H_{\sigma_2}\smallsetminus H_{\sigma_2}$.
Here we choose $\calB=\{a_1,a_4\}$ in Definition \ref{10} and observe that $a_2-a_1+a_4=a_7$.
However, one can verify as in Example \ref{23} that $a_2-a_1$ is not
an actual obstruction to the Cohen--Macaulayness of $\KK[H]$.
\end{rmk}

From \cite{Kou76} the following implications are known:
\begin{gather*}
\KK[H_\tau]\text{ is Cohen--Macaulay }\forall\tau\in\dot\Phi^{\boldone}_\calA\hspace{40ex}\\
\Longrightarrow\quad\gr^\calA(\KK[H])\text{ is Cohen--Macaulay}\hspace{0ex}\\
\hspace{40ex}\Longrightarrow\quad\KK[H]\text{ is Cohen--Macaulay}.
\end{gather*}
The examples in this section show that neither implication can be reversed in general. 
In particular, the proofs for \cite[Prop.~3.1]{Oku06a} and \cite[Prop.~3.4]{Oku06b} are faulty. However, the statements of these propositions themselves are correct by \cite[Cor.~9.2]{MMW05} which is based on vanishing theorems for Euler--Koszul homology on toric modules.

In the sequel we give a Newton filtration proof for \cite[Prop.~3.1]{Oku06a} and \cite[Prop.~3.4]{Oku06b}. 
This generalizes the approach of \cite[Thm.~1.2]{Ado99} in the case of normal $H$.

By positivity of $H$ there is a linear form $\ell\colon \ZZ H\to\ZZ$ such that $\ell(\calA)\subseteq\NN\smallsetminus\{0\}$. 

\begin{lem}\label{26}
For a suitable $q\in\NN\smallsetminus\{0\}$, the affine hyperplane $P_q=\ell^{-1}(q)$ intersects the one-dimensional faces of $C_H$ in elements of $H$. 
\end{lem}

\begin{proof}
Let $q=\lcm(\ell(\calA))>0$.
Then $\ell(qa/\ell(a))=q$ and $qa/\ell(a)\in\NN a$.
Thus $P_q\cap\NN a\ne\emptyset$ for all $a\in\calA$ and the claim follows.
\end{proof}

Note that natural multiples of the $q$ from Lemma~\ref{26} also satisfy the conclusion of the lemma. 
As $H$ is Noetherian, there is $\ell\colon \QQ H\to\QQ$ with $\ell(H)\geq 0$, satisfying Lemma \ref{26} with $q=1$, for which
\[
\calA=\{h\in H\mid\ell(h)\le 1\}
\]
is a finite set of generators of $H$.
Thus, $\Delta_\calA^\boldone=C_H\cap\ell^{-1}([0,1])$ and $\Phi^{\boldone,r-1}_{\calA}=\{C_H\cap\ell^{-1}(1)\}$ is a singleton.
In this case, Algorithm \ref{7} terminates trivially with $\calA'=\calA$ returning the Newton weight vector
\[
L=(\deg_{\calA}(a))_{a\in\calA}=(\ell(a))_{a\in\calA}.
\]
Then, by Lemma \ref{3},
\[
\gr^L(\KK[F_{\calA}]/I_{\calA})=\gr^{\calA}(\KK[H])\cong\KK[H].	
\]
For any other set $\calA'$ of generators for $\KK[H]$ with corresponding presentation $\KK[F_{\calA'}]\to \KK[H]$ one can define a positive weight vector $L'$ on $\calA'$ by $L'_{a'}=\deg_\calA(a')$. 
Then $\Phi^{L',r-1}_{\calA'}$ is the same singleton $\{C_H\cap\ell^{-1}(1)\}$ as $\Phi^{\boldone,r-1}_{\calA}$, $\gr^{L'}(\KK[F_{\calA'}]/I_{\calA'})\cong \KK[H]$. 

The preceding arguments can be summarized as follows.

\begin{prp}\label{29}
On any presentation \eqref{28} of a positive semigroup $H$, there is a strictly positive weight vector $L$ such that $\gr^L(\KK[F_\calA]/I_\calA)\cong\KK[H]$.\qed
\end{prp}

\medskip

We return to the $A$-hypergeometric system $M_A(\beta)$ defined in the introduction where $A$ is an 
integer $d\times n$ matrix and $\beta$ a complex $d$-vector.
The leading actor in the homological study of the $A$-hypergeometric
system is the \emph{Euler--Koszul complex} $K_\bullet(S_A,\beta)$
whose $0$-th homology module is $M_A(\beta)$ \cite{MMW05}.
For a $\ZZ^d$-graded $R_A$-module $N$, $K_\bullet(N,\beta)$ is by
definition the Koszul complex of the $d$ commuting left
$D_A$-endomorphisms $E-\beta$ on $D_A\otimes_{R_A}N$ defined by
$E_i\circ y=(E_i-\deg_i(y))y$ for all $\ZZ^d$-homogeneous $y\in M$ and
by $\CC$-linear extension.

\begin{cor}
Assume that the columns of $A\in\CC^{d\times n}$ span a positive affine semigroup $\NN A$ such that $\ZZ A=\ZZ^d$ and that $S_A=\CC[\NN A]$ is a Cohen--Macaulay ring. 
Then the Euler--Koszul complex $K_\bullet(S_A,\beta)$ is acyclic for all $\beta\in\CC^d$.
\end{cor}

\begin{proof}
By Proposition~\ref{29}, there is a strictly positive weight vector $L$ on $\del$ such that $\gr^L(R_A/I_A)\cong S_A$.
Choose a rational $c>0$ and set $\deg_L(x_j)=c-\deg_L(\del_j)$ for all $j$. 
One obtains an induced filtration on $D_A$ with $L_k(R_A)=L_k(D_A)\cap R_A$ satisfying $\gr^L(D_A)=\CC[x,\del]$, and so
\begin{equation}\label{54}
\gr^L(D_A\otimes_{R_A}S_A)=\CC[x]\otimes_\CC S_A.
\end{equation}
The $L$-symbols $E$ of the Euler operators $E-\beta$ form a system of parameters on \eqref{54}, cf.~\cite{GKZ89}.
As $S_A$ is Cohen--Macaulay by hypothesis, the Koszul complex of $E$ on \eqref{54} is a resolution. 
The latter being the $L$-graded Euler--Koszul complex $\gr^L(K_\bullet(S_A,\beta))$, acyclicity of the ungraded complex $K_\bullet(S_A,\beta)$ follows.
\end{proof}

\begin{rmk}
Note the ``limit case'' $c\to\infty$, which corresponds to the weight
filtration on $D_A$ given by $(\boldone,\boldzero)$.  This filtration
has been curiously overlooked as a an elementary tool to show
exactness of the Euler--Koszul complex for Cohen--Macaulay $S_A$.
\end{rmk}


\bibliographystyle{amsalpha}
\bibliography{cmtr}

\end{document}